\documentclass[11pt,noinfoline]{article}
\RequirePackage[aop,amsthm,amsmath,natbib]{imsart}
\usepackage{subfig}
\RequirePackage[OT1]{fontenc}
\usepackage{amsfonts,latexsym,amssymb}
\usepackage{graphics,graphicx}
\usepackage[colorlinks=true,linkcolor=black,citecolor=black,urlcolor=black]{hyperref}
\usepackage[T1]{fontenc}
\usepackage{amsfonts}
\usepackage{verbatim}
\usepackage{amsmath, amsthm, amssymb}
\usepackage{natbib}
\usepackage{rotating}

\allowdisplaybreaks 

\usepackage{graphics}
\startlocaldefs

\def\sec{\setcounter{equation}{0}\setcounter{figure}{0}}

\newtheorem{Lemma}{Lemma}[section]

\newtheorem{theorem}[Lemma]{Theorem}

\newtheorem{corollary}[Lemma]{Corollary}
\newtheorem{remark}[Lemma]{Remark}
\newtheorem{example}[Lemma]{Example}

\newcommand{\reals}{{\mathbb R}}
\newcommand{\bbr}{\reals}
\newcommand{\vep}{\varepsilon}

\def\definedas{\stackrel{\Delta}{=}}\newcommand{\BX}{{\bf X}}

\newcommand{\bt}{{\bf t}}
\newcommand{\bs}{{\bf s}}
\newcommand{\ba}{{\bf a}}
\newcommand{\bb}{{\bf b}}
\newcommand{\be}{{\bf e}}
\newcommand{\bnull}{{\bf 0}}
\newcommand{\bomega}{{\bf \omega}}
\newcommand{\one}{{\bf 1}}

\newcommand{\cov}{\text{cov}}

\newcommand{\cH}{{\mathcal H}}

\newcommand{\cC}{{\mathcal C}}

\def\:{:\,}
\def\definedby{\stackrel{\Delta}{=}}

\newcommand{\bc}{\begin{center}}
\newcommand{\ec}{\end{center}}
\newcommand{\beq}{\begin{eqnarray}}
\newcommand{\eeq}{\end{eqnarray}}
\newcommand{\beqq}{\begin{eqnarray*}}
\newcommand{\eeqq}{\end{eqnarray*}}

\endlocaldefs

\begin{document}
\begin{frontmatter}
\runtitle{Climbing down Gaussian  peaks}
\title{Climbing down Gaussian  peaks}   

\begin{aug}
\author{\fnms{Robert J.} \snm{Adler}\thanksref{t1,t4}
\ead[label=e1]{robert@ee.technion.ac.il}
\ead[label=u1,url]{webee.technion.ac.il/people/adler}}
\and 
\author{\fnms{Gennady}
  \snm{Samorodnitsky}\thanksref{t1,t2}\ead[label=e2]{gs18@cornell.edu}
\ead[label=u2,url]{www.orie.cornell.edu/$\sim$gennady/} }
\thankstext{t1}{Research supported in part by US-Israel Binational
Science Foundation, 2008262}
\thankstext{t2}{Research supported in part by ARO 
grant W911NF-12-10385 and NSF grant  DMS-1005903 at Cornell University.}
\thankstext{t4}{Research supported in part by  URSAT, ERC Advanced
  Grant 320422.}
 
\runauthor{Adler and Samorodnitsky}
\end{aug}

\begin{keyword}[class=AMS]
\kwd[Primary ]{60G15, 60F10; }
\kwd[Secondary ]{60G60, 60G70, 60G17.}
\end{keyword}\begin{keyword}
\kwd{Gaussian process, excursion set, large deviations, exceedence
  probabilities, topology.}
\end{keyword}

\setcounter{lemma}{0}


\begin{abstract}
How likely is the high level of a continuous Gaussian random field on
an Euclidean space  to
have a ``hole'' of a certain dimension and depth? Questions of this
type are difficult, but in this paper we make progress on questions
shedding new light in existence of holes. How likely is the field to
be above a high level on one compact set (e.g. a sphere) and to be
below a fraction of that level on some other compact set, e.g. at the
center of the corresponding ball? How likely is the field to be below that
fraction of the level {\it anywhere} inside the ball? We work on the
level of large deviations. 
\end{abstract}
\end{frontmatter}

\section{Introduction}\label{sec: Intro}
\sec

Let $T$ be a compact subset of $\bbr^d$. For a real-valued 
sample continuous random field $\BX = (X(\bt), \, \bt\in
T)$ and a level $u$, the
excursion set of $\BX$ above the level $u$ is the random set 
\begin{equation} \label{e:exc.set}
A_u= \bigl\{ \bt\in T:\, X(\bt)>u\bigr\}\,.
\end{equation}
Assuming that the entire index set $T$ has no interesting topological
features (i.e., $T$  is homotopic to a ball), 
what is the structure of the excursion  set? This is a generally difficult and
important question, and it constitutes an active research area. See
\cite{adler:taylor:2007} and \cite{azais:wschebor:2009} for in-depth
discussions. In this paper we consider the case when the random field
$\BX$ is Gaussian.  Even in this case the problem is still difficult. 

In a previous paper \cite{adler:moldavskaya:samorodnitsky:2014} we
studied a certain connectedness property of the excursion set $A_u$
for high level $u$. Specifically, given two distinct points in
$\bbr^d$, say, $\ba$ and $\bb$, we studied the asymptotic behaviour, 
as $u\to\infty$, of the conditional probability that, given $X(\ba)>u$ and
$X(\bb)>u$, there exists a path $\xi$ between $\ba$ and $\bb$ such
that $X(\bt)>u$ for every $\bt\in\xi$. 

In contrast, in this paper our goal is to study the probability that
the excursion set $A_u$  has holes of a certain size over which the random
field drops a fraction of the level $u$. We start with some examples
of the types of  probabilities we will look at. We will use the
following notation. For an Euclidean ball $B$ will denote by $c_B$ its
center and by $S_B=\partial(B)$ the sphere forming its boundary.
Consider the following probabilities.   For $0<r\leq 1$ denote 
\begin{equation} \label{e:spheres}
\Psi_{\text sp}(u; r) = P\left( \text{there exists a ball $B$
entirely in $T$ }\right.
\end{equation}
$$
\left. \text{   such that 
$X(\bt)>u$ for all $\bt\in S_B$ but
    $X(\bs)<r u$ for some $\bs\in B$}\right)
$$
and
\begin{equation} \label{e:spheres.center}
\Psi_{\text sp; c}(u; r) = P\left( \text{there exists a ball $B$
   entirely in $T$  } \right.
\end{equation}
$$
\left. \text{ such that $X(\bt)>u$ for all $\bt\in S_B$ but
    $X(c_B)<r u$}\right)\,.
$$
Simple arguments involving continuity show that the relevant sets in
both \eqref{e:spheres} and \eqref{e:spheres.center} are measurable. 
Therefore,  the probabilities $\Psi_{\text sp}(u; \tau) $ and $\Psi_{\text
  sp; c}(u; \tau) $ are well defined. These are the probabilities of
events that, for some ball, the boundary of the ball belongs to the
excursion set $A_u$, but the excursion set has a hole somewhere inside
the ball in one case, containing the center of the ball in another case, in
which the value of the field drops below $\tau u$. 

We study the logarithmic behaviour of probabilities of this type by
using the large deviation approach. We start with a setup somewhat
more general than that described above. Specifically, let $\cC$
be a collection of ordered pairs $(K_1,K_2)$ of 
nonempty compact subsets of $T$. We denote, for
$0<r\leq 1$, 
\begin{equation} \label{e:two.sets}
\Psi_{\cC}(u; r) =  P\left(  \text{ there is $(K_1,K_2)\in \cC$
    such that  }\right.
\end{equation} 
$$
\left. \text{  $X(\bt)>u$ for each $\bt\in K_1$  and $X(\bt)<r u$ for each
    $\bt\in K_2$.}\right)
$$

We note that the probabilities  $\Psi_{\text sp}(u;
r) $ and $\Psi_{\text   sp; c}(u; r) $ are special cases of the
probability 
$\Psi_{\cC}(u; r)$ with the collections $\cC$ being, respectively,
$$
\cC = \Bigl\{ \bigl( S_B, \bs\bigr), \ \text{ $B$ a ball entirely in
$T$ and $\bs\in B$}\Bigr\}
$$
and
$$ 
\cC = \Bigl\{ \bigl( S_B, c_B\bigr), \ \text{ $B$ a ball entirely in 
$T$ }\Bigr\}\,.
$$

In Section \ref{sec:large.deviations}  
we first introduce the necessary technical
background, and then prove a large deviation result in
the space of continuous functions for the probability $\Psi_{\cC}(u;
r)$. This result establishes a connection of the asymptotic behaviour
of  the probability $\Psi_{\cC}(u;r)$ to a certain optimization problem. The dual
formulation of this problem involves optimization over a  
family of probability measures, and in Section
\ref{sec:optimal.measures} we describe important properties of the
measures that are optimal for the dual problem. The general theory
developed in these two  sections
leads to particularly transparent and intuitive results when applied
to isotropic Gaussian fields. This is explored in Section
\ref{sec:isotropic}.

\section{A large deviations result} \label{sec:large.deviations}

Consider a real-valued centered continuous Gaussian
random field indexed by a compact subset $T\subset \bbr^d$, 
$\BX = (X(\bt), \, \bt\in T)$. We denote the covariance function of
$\BX$ by  $R_\BX(\bs,\bt)= \cov(X(\bs),X(\bt))$. We view $\BX$ as a
Gaussian random element in the space $C(T)$ of continuous functions 
on $T$, equipped with the supremum norm, whose law is  a Gaussian
probability measure $\mu_\BX$ on $C(T)$. See 
e.g. \cite{vandervaart:vanzanten:2008} about this change of the
viewpoint, and for more information on the subsequent discussion. 

The reproducing kernel Hilbert space (henceforth RKHS)
$\cH$ of the Gaussian measure $\mu_\BX$ (or of the random field $\BX$)
is a subspace of  $C(T)$  obtained as
follows. We identify $\cH$ with the closure $\mathcal L$
in the mean square norm of
the space of finite linear combinations $\sum_{j=1}^k a_jX(\bt_j)$ of
the values of the process, $a_j\in\bbr, \, \bt_j\in T$  for
$j=1,\ldots, k$, $k=1,2,\ldots$  
via the injection ${\mathcal L}\to C(T)$ given by 
\begin{equation} \label{e:embed.Gen}
H\to w_H = \Bigl( E\bigl( X(\bt)H\bigr), \ \bt\in T \Bigr)\,.
\end{equation}
We  denote by $(\cdot,\cdot)_\cH$ and $\| \cdot\|_\cH$ the inner
product and the norm in the  RKHS $\cH$. By definition, 
\beq
\label{norm:equalities}
\|w_H\|^2_\cH \ = \ E(H^2)\,.
\eeq

The ``reproducing property'' of the space $\cH$ is a consequence of
the following observations. For every $\bt\in\bbr^d$, the fixed $\bt$
covariance 
function $R_{\bt}=R(\cdot,\bt)$ is in $\cH$. Therefore, for every $w_H\in
\cH$, and $\bt\in\bbr^d$, $w_H(\bt)=(w_H,R_\bt)_\cH$. In particular,
the coordinate projections are continuous operations on the RKHS.

The quadruple $(C(T), \cH, w, \mu_\BX)$  is  a Wiener
quadruple in the sense of Section 3.4 in
\cite{deuschel:stroock:1989}. This allows one to use the machinery of
large deviations for Gaussian measures described there.

The following result is a straightforward application of the general large
deviations machinery.
\begin{theorem} \label{t:LDP.path} 
\ Let $\BX = (X(\bt), \, \bt\in T)$ be a continuous Gaussian
random field on a compact set $T\subset\bbr^d$. Let $\cC$
be a collection of ordered pairs $(K_1,K_2)$ of 
nonempty compact subsets of $\bbr^d$, compact in the product Hausdorff
distance. Then for $0<r\leq 1$, 
\beq
\label{e:LD.limit} 
-\frac12 \lim_{\tau\uparrow r}D_{\cC}(\tau)  \leq 
\liminf_{u\to\infty} \frac{1}{u^2} \log \Psi_{\cC}(u;  r) 
\eeq
\beqq
\leq \limsup_{u\to\infty} \frac{1}{u^2} \log \Psi_{\cC}(u;  r) 
\leq -\frac12 D_{\cC}(r)\,,
\eeqq
where for $r>0$, 
\beq \label{e:optimization.big}
D_{\cC}(r)
 &\definedas&\inf\left\{
EH^2 \: H\in {\mathcal L},\
   \text{and, for some  $(K_1,K_2)\in \cC$, }\right.
\eeq
\beqq
\left. \text{
  $E\bigl(X(\bt)H\bigr)\geq 1$ for each $\bt\in K_1$  and $E\bigl(
X(\bt)H\bigr)\leq r $ for each
    $\bt\in K_2$}  \right\}.
\eeqq
\end{theorem}
\begin{proof}
As is usual in large deviations arguments, we write,  for $u>0$, 
$$
\Psi_{\cC}(u;  r)  =
P\bigl( u^{-1}\BX \in A\bigr)\,, 
$$
where $A$ is the open subset of $C(T)$  given by
$$
A \definedby \Bigl\{ \bomega\in C(T):\  \text{ there is $(K_1,K_2)\in \cC$
    such that  }
$$
$$
 \text{  $\omega(\bt)>1$ for each $\bt\in K_1$  and $\omega(\bt)<r $ for each
    $\bt\in K_2$}  \Bigr\}.  
$$ 
We use Theorem 3.4.5 in \cite{deuschel:stroock:1989}. We have 
\begin{equation} \label{e:LD.estimate}
-\inf_{\bomega\in A} I(\bomega) \leq \liminf_{u\to\infty}
\frac{1}{u^2} \log \Psi_{\cC}(u;  \tau)
\leq \limsup_{u\to\infty}
\frac{1}{u^2} \log \Psi_{\cC}(u;  \tau)
\leq -\inf_{\bomega\in \bar A}
I(\bomega)\,. 
\end{equation}
By Theorem 3.4.12 of
\cite{deuschel:stroock:1989}  the rate function $I$  can be written as
\begin{equation} \label{e:rate.function}
I(\bomega) = \left\{ \begin{array}{ll}
\frac12 \| \bomega\|_\cH^2 & \text{if $\bomega\in \cH$}, \\
\infty  & \text{if $\bomega\notin \cH$,}
\end{array} \right.
\end{equation}
for $\bomega\in C(T)$.  Since $\cC$ is compact in the product Hausdorff distance, 
$$
\bar A \subseteq \Bigl\{ \bomega\in C(T):\  \text{ there is $(K_1,K_2)\in \cC$
    such that  }
$$
$$
 \text{  $\omega(\bt)\geq 1$ for each $\bt\in K_1$  and
   $\omega(\bt)\leq r $ for each     $\bt\in K_2$}  \Bigr\}, 
$$ 
and so  \eqref{e:LD.estimate} already
contains the upper  limit statement in \eqref{e:LD.limit}. Further, 
for any $0<\vep<1$,
\beqq
\inf_{\bomega\in  A} I(\bomega) \leq \inf\left\{
EH^2 \: H\in {\mathcal L},\
   \text{and, for some  $(K_1,K_2)\in \cC$, }\right.
\eeqq
\beqq
\left. \text{
  $\omega_H(\bt)\geq 1+\vep$ for each $\bt\in K_1$  and
  $\omega_H(\bt)\leq (1-\vep)r $ for each 
    $\bt\in K_2$}  \right\}
\eeqq
$$
= (1+\vep)^2 D_{\cC}\left( \frac{1-\vep}{1+\vep}r\right)\,.
$$
Letting $\vep\downarrow 0$ establishes 
the lower limit statement in \eqref{e:LD.limit}.
\end{proof}

The lower bound in \eqref{e:LD.limit} can be strictly smaller than the
upper bound, as the following example shows. 
 We will see in the sequel that in certain cases of interest the two
 bounds do coincide. 
\begin{example} \label{ex:not.equal}
{\rm 
Let $T=\{ 0,1,2\}$. Starting with independent standard normal random
variables $Y_1, Y_2$ we define, for $0<r_0<1$ and $\sigma>r_0$, 
$$
X(0)=Y_1, \ X(1) = r_0Y_1, \ X(2) = \sigma Y_1+Y_2\,.
$$
Note that in this case ${\mathcal L} = \{ a_1Y_1+a_2Y_2, \
a_1\in\bbr, \, a_2\in\bbr\}$. 

Let $\cC = \bigl\{ \bigl( \{ 0\}, \{ 1\}\bigr), \, \bigl( \{ 0\}, \{
2\}\bigr)\bigr\}$. It is elementary to check that 
$$
D_{\cC}(r) = \left\{
\begin{array}{ll}
1+ (\sigma-r)^2 & \text{for $0<r<r_0$,} \\
1 & \text{for $r\geq r_0$,}
\end{array}
\right.
$$
and that this function is not left continuous at $r=r_0$. 
}
\end{example}

For a fixed pair $(K_1,K_2)\in \cC$ denote 
\begin{equation} \label{e:fixed.pair}
D_{K_1,K_2}(r)  =\inf\left\{
EH^2 \: H\in {\mathcal L}\
   \text{such that} \right.
\end{equation}
$$
\left. \text{
  $E\bigl(X(\bt)H\bigr)\geq 1$ for each $\bt\in K_1$  and $E\bigl(
X(\bt)H\bigr)\leq r $ for each
    $\bt\in K_2$}  \right\}.
$$
Clearly,
\begin{equation} \label{e:connect.R}
D_{\cC}(r) = \min_{(K_1,K_2)\in \cC} D_{K_1,K_2}(r) \,,
\end{equation}
with the minimum actually achieved. Furthermore, an application of
Theorem \ref{t:LDP.path} to the case of $\cC$ consisting of a single
ordered pair of sets immediately shows that
\beq \label{e:ldp:single}
-\frac12 \lim_{\tau\uparrow r}D_{K_1,K_2}(\tau)  &&\leq 
\liminf_{u\to\infty} \frac{1}{u^2} \log \Psi_{K_1,K_2}(u;  r)
\nonumber \\
&&\leq \limsup_{u\to\infty} \frac{1}{u^2} \log \Psi_{K_1,K_2}(u;  r) 
\leq -\frac12 D_{K_1,K_2}(r)\,,
\eeq
where
$$
\Psi_{K_1,K_2}(u;r) = P\left( \text{  $X(\bt)>u$ for
    all $\bt\in K_1$  and $X(\bt)<r u$ for all  $\bt\in
    K_2$.}\right) 
$$

The next result describes  useful properties of the function $D_{K_1,K_2}$. 
\begin{theorem} \label{t:dual.repr} 
  (a)\ If the feasible set in \eqref{e:fixed.pair} is not empty, then
the infimum is achieved, at a unique $H\in {\mathcal L}$. 

(b) \ The following holds true:
\begin{equation} \label{e:dual}
D_{K_1,K_2}(r) = \left\{ \min \left[ \min_{\mu_1\in M_1^+(K_1)}
\int_{K_1}\int_{K_1} R_\BX( \bt_1,\bt_2)\, \mu_1(d\bt_1)\mu_1(d\bt_2), \
\right. \right.
\end{equation}
$$
\left. \left.
\min_{\mu_1\in M_1^+(K_1), \mu_2\in M_1^+(K_2) \atop 
\text{\rm subject to \eqref{e:cond.1}}}
\frac{A_{K_1,K_2}(\mu_1,\mu_2)}{B_{K_1,K_2}(\mu_1,\mu_2; r)}
\right]\right\}^{-1}
$$
with 
$$
A_{K_1,K_2}(\mu_1,\mu_2) =
$$
$$
 \int_{K_1}\int_{K_1} R_\BX( \bt_1,\bt_2)\,
 \mu_1(d\bt_1)\mu_1(d\bt_2)  \int_{K_2}\int_{K_2} R_\BX( \bt_1,\bt_2)\,
\mu_2(d\bt_1)\mu_2(d\bt_2) 
$$
$$
 - \left( \int_{K_1}\int_{K_2} R_\BX(
\bt_1,\bt_2)\, \mu_1(d\bt_1)\mu_2(d\bt_2)\right)^2,
$$

$$
B_{K_1,K_2}(\mu_1,\mu_2; r)= 
$$
$$
r^2 \int_{K_1}\int_{K_1} R_\BX( \bt_1,\bt_2)\,
\mu_1(d\bt_1)\mu_1(d\bt_2) -2r \int_{K_1}\int_{K_2} R_\BX(
\bt_1,\bt_2)\, \mu_1(d\bt_1)\mu_2(d\bt_2)
$$
$$
 +  \int_{K_2}\int_{K_2} R_\BX( \bt_1,\bt_2)\,
\mu_2(d\bt_1)\mu_2(d\bt_2),
$$
and the condition in the minimization problem is 
\begin{equation} \label{e:cond.1}
\int_{K_1}\int_{K_2} R_\BX(
\bt_1,\bt_2)\, \mu_1(d\bt_1)\mu_2(d\bt_2) \geq r  \int_{K_1}\int_{K_1}
R_\BX( \bt_1,\bt_2)\, 
\mu_1(d\bt_1)\mu_1(d\bt_2)\,.
\end{equation}
\end{theorem}
\begin{proof}
For part (a), let $(H_n)\subset {\mathcal L}$ be a sequence of
elements satisfying the constraints in \eqref{e:fixed.pair} such that 
$EH_n^2\to D_{K_1,K_2}(r)$ as $n\to\infty$. By the Banach - Alaoglu
theorem (see e.g. Theorem 2, p. 424 in  
\cite{dunford:schwartz:1988}), the sequence
$(H_n)$ is weakly relatively compact in $\mathcal L$, and so there is
$H\in {\mathcal L}$  and a subsequence $n_k\to\infty$ 
such that $E(H_{n_k}Y)\to E(HY)$ as $k\to\infty$ for each $Y\in
{\mathcal L}$. Further, $EH^2\leq \liminf_{k\to\infty} EH_{n_k}^2=
D_{K_1,K_2}(r)$. Therefore, $H$ is an optimal solution to the problem
\eqref{e:fixed.pair}. The uniqueness of $H$ follows from the convexity
of the norm. 

For part (b) we will  use the Lagrange duality approach of Section 8.6 in
\cite{luenberger:1969}. Let ${\tt Z}=C(K_1) \times C(K_2)$, which we
equip with the norm $\| (\varphi_1,\varphi_2)\|_{\tt Z} = \max\bigl(
\|\varphi_1\|_{C(K_1)}, \|\varphi_2\|_{C(K_2)}\bigr)$. Consider the 
closed convex cone in $\tt Z$ defined by ${\tt P} = \{
(\varphi_1,\varphi_2):\, \varphi_i(t) \geq 0$ for all $t\in K_i, \
i=1,2\}$. Its dual cone, which is a subset of ${\tt Z}^*$, can be
identified with $M^+(K_1)\times M^+(K_2)$, under the action 
$$
(\mu_1,\mu_2)\bigl( (\varphi_1,\varphi_2)\bigr) = \int_{K_1} \varphi_1\, d\mu_1+
 \int_{K_2} \varphi_2\, d\mu_2
$$
for a finite measure $\mu_i$ on $K_i$, $i=1,2$. Define a convex
mapping $G:\, {\mathcal L} \to   {\tt Z}$ by
$$
G(H) =   \Bigl( \bigl( 1-w_H(\bt),\, \bt\in K_1\bigr), \ \bigl(
  w_H(\bt)-r, \, \bt\in K_2\bigr)\Bigr).
$$
We can write
\begin{equation} \label{e:primal}
\bigl( D_{K_1,K_2}(r)\bigr)^{1/2} =\inf\left\{ (EH^2)^{1/2}\, : H\in {\mathcal L}, \ G(H)\in
  -{\tt P}\right\}\,.
\end{equation} 

We start with the assumption that the feasible set in
\eqref{e:fixed.pair} and \eqref{e:primal} is not empty. Let $z>r$, and
consider the optimization problems \eqref{e:fixed.pair} and
\eqref{e:primal} for $D_{K_1,K_2}(z)$. The feasible set in these
problems has now an interior point, and in this case
Theorem 1 (p. 224) in \cite{luenberger:1969} applies. 
We conclude that
\begin{equation} \label{e:optimize.lagrange}
\bigl( D_{K_1,K_2}(z)\bigr)^{1/2} = \max_{\mu_1\in M^+(K_1),\, 
  \mu_2\in M^+(K_2)}
\inf_{H\in {\mathcal L}} \biggl[ (EH^2)^{1/2} 
\end{equation}
$$
+ \int_{K_1} \bigl(
  1-w_H(\bt)\bigr)\, \mu_1(d\bt) + \int_{K_2} \bigr( w_H(\bt)-z\bigr)\,
  \mu_2(d\bt)\biggr]\,,
$$
and the ``max'' notation is legitimate, because the maximum is, in
fact, achieved. For $i=1,2$ and $\mu_i\in M^+(K_i)$ denote by
$\|\mu_i\|$ its total mass, and by $\hat \mu_i\in M_1^+(K_i)$ the
normalized measure $\hat \mu_i=\mu_i/\|\mu_i\|$ (if $\|\mu_i\|=0$, we
use for $\hat \mu_i$ an arbitrary fixed probability measure in
$M^+(K_i)$). Then 
$$
\bigl( D_{K_1,K_2}(z)\bigr)^{1/2} = \max_{\mu_1\in M^+(K_1),\, 
  \mu_2\in M^+(K_2)} \biggl\{ \|\mu_1 \| -z \|\mu_2\|
$$
$$
+ \inf_{H\in {\mathcal L}} \left[ (EH^2)^{1/2} 
- \|\mu_1 \|\int_{K_1} 
 w_H(\bt) \, \hat \mu_1(d\bt) + \|\mu_2\|\int_{K_2}   w_H(\bt) \,
 \hat \mu_2(d\bt)\right]\biggr\}\,.
$$
Note that for fixed $\mu_i\in M^+(K_i)$, $i=1,2$ we have
$$
\inf_{H\in {\mathcal L}} \left[ (EH^2)^{1/2} 
- \|\mu_1 \|\int_{K_1} 
 w_H(\bt) \, \hat \mu_1(d\bt) + \|\mu_2\|\int_{K_2}   w_H(\bt) \,
 \hat \mu_2(d\bt)\right]
$$
$$ 
= \inf_{a\geq 0} a\biggl\{ 1-\sup_{H\in {\mathcal L},\, EH^2=1} 
\left[ \|\mu_1 \|\int_{K_1} 
 w_H(\bt) \, \hat \mu_1(d\bt) - \|\mu_2\|\int_{K_2}   w_H(\bt) \,
 \hat \mu_2(d\bt)\right]\biggr\}
$$
$$
= \left\{ \begin{array}{ll}
0 & \text{if} \ \sup_{H\in {\mathcal L},\, EH^2=1}  [ \ldots
 ]\leq 1 \\
-\infty & \text{if} \ \sup_{H\in {\mathcal L},\, EH^2=1}  [ \ldots
 ]> 1 
\end{array} \right..
$$
Therefore,
$$
\begin{array}{ll}
\bigl( D_{K_1,K_2}(z)\bigr)^{1/2} = &
\max_{\mu_1\in M^+(K_1),\, 
  \mu_2\in M^+(K_2)}   \ \ \bigl(  \|\mu_1 \| -z \|\mu_2\|\bigr) \\
\\ 
 &\text{subject to}
\end{array}
$$
$$
 \sup_{H\in {\mathcal L},\, EH^2=1} 
\left[ \|\mu_1 \|\int_{K_1} 
 w_H(\bt) \, \hat \mu_1(d\bt) - \|\mu_2\|\int_{K_2}   w_H(\bt) \,
 \hat \mu_2(d\bt)\right] \leq 1\,. 
$$
Note that by the reproducing property, for fixed $\mu_1\in M^+(K_1),\, 
  \mu_2\in M^+(K_2)$,
$$
\sup_{H\in {\mathcal L},\, EH^2=1} 
\left[ \|\mu_1 \|\int_{K_1} 
 w_H(\bt) \, \hat \mu_1(d\bt) - \|\mu_2\|\int_{K_2}   w_H(\bt) \,
 \hat \mu_2(d\bt)\right] 
$$
$$ 
= \sup_{w\in \cH,\, \|w\|_\cH=1} \left( w, \,
\|\mu_1 \|\int_{K_1} 
 R_\bt(\cdot) \, \hat \mu_1(dt) - \|\mu_2\|\int_{K_2}  R_\bt(\cdot) \,
 \hat \mu_2(dt) \right) _\cH \,.
$$
Assuming that the element in the second position in the inner product
is nonzero, the supremum is achieved at that element scaled to have a
unit norm. Therefore, value of the supremum is
$$
\left\| \|\mu_1 \|\int_{K_1} 
 R_\bt(\cdot) \, \hat \mu_1(dt) - \|\mu_2\|\int_{K_2}  R_\bt(\cdot) \,
 \hat \mu_2(dt) \right\|_\cH\,,
$$
which is also trivially the case if the element in the second position in
the inner product is the zero element. In any case, using the
definition of the norm in $\cH$, we conclude that
$$
\bigl( D_{K_1,K_2}(z)\bigr)^{1/2} = 
\max_{m_1\geq 0, \, m_2\geq 0}\max_{\mu_1\in M_1^+(K_1),\, 
  \mu_2\in M_1^+(K_2)}   \ \ \bigl(  m_1 -z\, m_2\bigr) 
$$
\begin{equation} \label{e:Rz.dual}
\hskip -1.1in \text{subject to}
\end{equation}
$$
m_1^2 \int_{K_1}\int_{K_1} R_\BX( \bt_1,\bt_2)\,
\mu_1(d\bt_1)\mu_1(d\bt_2) -2m_1m_2 \int_{K_1}\int_{K_2} R_\BX(
\bt_1,\bt_2)\, \mu_1(d\bt_1)\mu_2(d\bt_2)
$$
$$
+ m_2^2 \int_{K_2}\int_{K_2} R_\BX( \bt_1,\bt_2)\,
\mu_2(d\bt_1)\mu_2(d\bt_2) \leq 1\,.
$$

Next, we show that \eqref{e:Rz.dual} holds for $z=r$ as well. Let
$A(z), \ z\geq r$ be the value of the maximum in the right hand side
of \eqref{e:Rz.dual}. We know that $D_{K_1,K_2}(z) = A(z)^2$ for
$z>r$. Moreover, it is clear that $A(z)\uparrow A(r)$ as $z\downarrow
r$. Therefore, in order to extend \eqref{e:Rz.dual} to $z=r$ it is
enough to prove that 
\begin{equation} \label{e:R.rcts}
\lim_{z\downarrow r} D_{K_1,K_2}(z) =D_{K_1,K_2}(r)\,.
\end{equation}
To this end, choose a sequence $z_n\downarrow r$. For $n\geq 1$ there
is, by part (a), the optimal solution $H_n$ of the problem
\eqref{e:fixed.pair} corresponding to $z_n$.  
Appealing to the  Banach-Alaoglu theorem, we see that the sequence
$(H_n)$ is weakly relatively compact in $\mathcal L$, and so there is
$H\in {\mathcal L}$  to which it converges weakly along a subsequence. This $H$ is,
clearly, feasible in \eqref{e:fixed.pair} for $z=r$.  Furthermore, 
$EH^2\leq  \lim_{n\to\infty } D_{K_1,K_2}(z_n)$, implying that
$D_{K_1,K_2}(r)\leq  \lim_{z\downarrow r}
D_{K_1,K_2}(z)$, thus giving us  the only nontrivial
inequality in \eqref{e:R.rcts}. Therefore, \eqref{e:Rz.dual} holds for
$z=r$. 

A part of the optimization problem in \eqref{e:Rz.dual} with $z=r$ has
the form
$$
\max_{m_1\geq 0, \, m_2\geq 0} \ \ \bigl(  m_1 -r\, m_2\bigr) 
$$
\begin{equation} \label{e:sep.prob}
\hskip -1.1in \text{subject to}
\end{equation}
$$
am_1^2 - 2bm_1m_2 + cm_2^2\leq 1
$$
for fixed numbers $a\geq 0, \, c\geq 0$ and $b\in \bbr$. In our case,
$$
a= \int_{K_1}\int_{K_1} R_\BX( \bt_1,\bt_2)\,
\mu_1(d\bt_1)\mu_1(d\bt_2),  \ b = \int_{K_1}\int_{K_2} R_\BX(
\bt_1,\bt_2)\, \mu_1(d\bt_1)\mu_2(d\bt_2)
$$
and 
\begin{equation} \label{e:abc}
c= \int_{K_2}\int_{K_2}
R_\BX( \bt_1,\bt_2)\, \mu_2(d\bt_1)\mu_2(d\bt_2) \,.
\end{equation}
These specific numbers satisfy the condition
\begin{equation} \label{e:abc}
b^2\leq ac\,,
\end{equation}
and we will assume that this condition holds in the problem
\eqref{e:sep.prob} we will presently consider. 

As a first step, it is clear that replacing the inequality constraint
in this problem by the equality constraint
$$
a_2m_1^2 - 2bm_1m_2 + cm_2^2=1
$$
does not change the value of the maximum, so we may work with the
equality constraint instead. The resulting problem can be easily
solved, e.g. by checking the boundary values $m_1=0$ or $m_2=0$, and
using the Lagrange multipliers if both $m_1>0$ and $m_2>0$. The
resulting value of the maximum in this problem is
\begin{equation} \label{e:max.value} 
\begin{array}{ll}
a^{-1/2} & \text{if} \ \ b\leq r a  \\
\left(\frac{c+ r^2a - 2 r b}{ac-b^2}\right)^{1/2} & \text{if} \ \
b> r a  
\end{array}.
\end{equation}
Moreover, it is elementary to check that we always have
$$
\left(\frac{c+r^2a - 2r b}{ac-b^2}\right)^{1/2} \geq
\frac{1}{a^{1/2}}\,.
$$
Substituting \eqref{e:max.value} into \eqref{e:Rz.dual} with $z=r$ and
using the values of $a, b, c$ given in \eqref{e:abc} gives us the
representation \eqref{e:dual}. 

It remains to consider the case when the feasible set in
\eqref{e:fixed.pair} and \eqref{e:primal} is empty. In this case
$D_{K_1,K_2}(r)=\infty$, so we need to prove that the optimal value in
the dual problem \eqref{e:Rz.dual} (with  $\max$ replaced by 
$\sup$ in the statement) is infinite as well. For this purpose we use
the idea of subconsistency in Section 3 of \cite{anderson:1983}. We
write the minimization problem \eqref{e:fixed.pair} as a linear
program with conic constraints, called $IP$ in that paper, with the
following 
parameters. The space $X=\bbr\times {\mathcal L}$ is in duality with
the itself, $Y=X$. The space $Z=C(K_1) \times C(K_2)$ (as above) is in
duality with the space $W=M(K_1)\times M(K_2)$, the product of the
appropriate spaces of finite signed measures. The vector $c\in Y$ has
the unity as its $\bbr$ element, and the zero function as its
$\mathcal L$ element. The function $A:\, X\to Z$ is given by
$$
A(\alpha, H) = \Bigl( \bigl( E(HX(\bt), \, \bt\in K_1\bigr), \, \bigl(
E(HX(\bt)), \, \bt\in K_1\bigr)\Bigr), \ \alpha\in \bbr, \, H\in {\mathcal
  L}\,.
$$
The vector $b\in Z$ is given by a pair of continuous functions, the
first one takes the constant value of $1$ over $K_1$, while the second
one takes the constant value of $r$ over $K_2$. The positive cone $Q$
in $Z$ is defined by $Q= C_+(K_1) \times (-C_+(K_2))$, where
$C_+(K_i)$ is the subset of $C(K_i)$ consisting of nonnegative
functions, $i=1,2$. Finally, the positive cone $P$ in $X$ is defined
by 
$$
P=\bigl\{ (\alpha, H):\, \alpha\geq (EH^2)^{1/2}\bigr\}\,.
$$
It is elementary to verify that the dual problem $IP^*$ of
\cite{anderson:1983} coincides with
the maximization problem \eqref{e:Rz.dual}. 

Note that the dual problem is consistent  (has a feasible
solution). By Theorem 3 in Section 3 of \cite{anderson:1983} (see a
discussion at the end of that section), in order to prove that the
optimal value of the dual problem is infinite, we need to rule out the
possibility that the original (primal) problem is subconsistent with a
finite subvalue. With a view of obtaining a contradiction, assume the
subconsistency with a finite subvalue of the primal problem. Then
there are sequences $(x_n)\subset P$ and $(z_n)\subset Q$ such that 
$Ax_n-z_n\to b$ as $n\to\infty$ and the sequence of evaluations
$(c,x_n)$ is bounded from above. With the present parameters, this
means that there is a sequence $(H_n)\subset {\mathcal L}$ with the
bounded sequence $(EH_n^2)$ of the second moments and two sequences of
functions $(\varphi_{i,n})\subset C_+(K_i)$, $i=1,2$, such that,
weakly, 
$$
\bigl( E(H_nX(\bt))-\varphi_{1,n}(\bt), \, \bt\in K_1\bigr) \to \bigl( 1, \,
\bt\in K_1\bigr) \,,
$$
$$
\bigl( E(H_nX(\bt))+\varphi_{2,n}(\bt), \, \bt\in K_2\bigr) \to \bigl( r, \,
\bt\in K_2\bigr) 
$$
as $n\to\infty$, with the obvious notation for constant
functions. Appealing, once again, to  the Banach-Alaoglu theorem, we
find that there is $H\in {\mathcal L}$ such that, along a subsequence,
$H_n\to H$ weakly. Since weak convergence implies pointwise
convergence, we immediately conclude that $E\bigl(X(\bt)H\bigr)\geq 1$
for each $\bt\in K_1$  and $E\bigl( X(\bt)H\bigr)\leq r $ for each
$\bt\in K_2$, contradicting the assumption that the feasible set
\eqref{e:fixed.pair} is empty. The obtained contradiction completes the proof. 
\end{proof}

\begin{remark} \label{rk:m1m2}
{\rm
It is an easy calculation to verify that, in the optimization problem
\eqref{e:sep.prob}, the optimal solution $(m_1,m_2)$ has the following
properties. In the case $b\leq r a$ in \eqref{e:max.value}  one has
$m_2=0$, whereas if $b>ra$ in
\eqref{e:max.value}, then the numbers $m_1$ and $m_2$ are both
positive, and 
$$
\frac{m_1}{m_2} = \frac{rb-c}{ra-b}\,.
$$
We will find these properties useful in the sequel. 
}
\end{remark}

\begin{remark} \label{rk:cty}
{\rm
We saw in Example \ref{ex:not.equal} that the function $D_{\cC}$ does
not, in general, need to be  continuous. However, the arguments
used in the proof of Theorem \ref{t:dual.repr}, together with the
compactness in 
the product Hausdorff distance of the set $\cC$, show that this
function is always right continuous

For a fixed pair
$(K_1,K_2)\in \cC$ even the absence of left continuity  for the function
$D_{K_1,K_2}$ is, in a sense, an exception and not the rule. Left
continuity is trivially true at any $r_0$ for 
which the minimization problem \eqref{e:fixed.pair} is infeasible. If
that problem is feasible, and it remains feasible for some $r<r_0$,
then the left continuity at $r_0$ still holds. To see this, suppose
$r_n\uparrow r_0$ as $n\to\infty$ is such that for some $0<\vep<\infty$ 
\begin{equation} \label{e:prove.cty}
\lim_{n\to\infty} \bigl( D_{K_1,K_2}(r_n)\bigr)^{1/2}=
\bigl( D_{K_1,K_2}(r_0) \bigr)^{1/2}+\vep\,.
\end{equation}
 Let $H_n$
be optimal in \eqref{e:fixed.pair} for $r_n, \, n\geq 1$, and $H$ be
optimal for $r_0$. Define $\hat H_n= (H_n+H)/2$. Then, for some sequence
$k_n\to\infty$, $\hat H_n$ is feasible in \eqref{e:fixed.pair} for
$r_{k_n}$, and 
$$
\bigl( E\hat H_n^2\bigr)^{1/2} \leq \Bigl( \bigl( E H_n^2\bigr)^{1/2}  +
\bigl( E H^2\bigr)^{1/2} \Bigr)/2\,.
$$
Letting $n\to\infty$ we obtain 
$$
\limsup_{n\to\infty} \bigl( E\hat H_n^2\bigr)^{1/2} \leq \bigl(
D_{K_1,K_2}(r_0) \bigr)^{1/2}+\vep/2\,,
$$
which contradicts \eqref{e:prove.cty}. Hence the left continuity at $r_0$.

Left continuity fails at a point $r_0$ at which the minimization
problem \eqref{e:fixed.pair} is feasible, but becomes infeasible at
any $r<r_0$. An easy modification of Example \ref{ex:not.equal} can be
used to exhibit such a situation. 
}
\end{remark} 

As long as one is not in the last situation described in the example,
it follows from \eqref{e:ldp:single} that
$$
\lim_{u\to\infty} \frac{1}{u^2} \log \Psi_{K_1,K_2}(u;  r) 
= -\frac12 D_{K_1,K_2}(r)\,.
$$
In this connection there is a very natural interpretation of the
structure of the representation \eqref{e:dual} of $
D_{K_1,K_2}(r)$. Notice that 
$$
\lim_{u\to\infty} \frac{1}{u^2} \log P\left( \text{  $X(\bt)>u$ for
    all $\bt\in K_1$} \right)
$$
$$
= -\frac12 \left\{  \min_{\mu_1\in M_1^+(K_1)}
\int_{K_1}\int_{K_1} R_\BX( \bt_1,\bt_2)\, \mu_1(d\bt_1)\mu_1(d\bt_2)
\right\}^{-1}\,.
$$
This can be read off part (b) in Theorem \ref{t:dual.repr}, and it is 
also a simple extension of the results in
\cite{adler:moldavskaya:samorodnitsky:2014}. Therefore, we can
interpret the situation in which the first minimum in the right hand side
of \eqref{e:dual} is the smaller of the two minima, as implying that the
order of magnitude of the probability $\Psi_{K_1,K_2}(u;  r)$ is
determined, at least at the logarithmic level, by the requirement 
that $X(\bt)>u$ for all $\bt\in K_1$. In this case, the requirement 
that $X(\bt)<ru$ for all $\bt\in K_2$ does not change the logarithmic
behaviour of the probability. This is not be entirely unexpected
since the normal random variables in the set $K_2$ ``prefer'' not to
take very large values. 

On the other hand, if the correlations between the variables of the
random field  in the set $K_1$ and those in the set $K_2$ are
sufficiently strong, it may happen that, once it is true that
$X(\bt)>u$ for each $\bt\in K_1$, the correlations will make it
unlikely that we also have $X(\bt)<ru$ for all $\bt\in K_2$. In that
case the second minimum in the right hand side
of \eqref{e:dual} will be the smaller of the two minima.

The discussion in Example \ref{rk:cty} also leads to the following
conclusion of Theorem \ref{t:LDP.path}.
\begin{corollary} \label{c:LDP.precise}
Under the conditions of Theorem \ref{t:LDP.path}, suppose that there
is $(K_1^{(r)}, K_2^{(r)})\in \cC$ such that 
$$
D_{\cC}(r) = D_{K_1^{(r)},K_2^{(r)}}(r)<\infty\,, 
$$
and such that the optimization problem \eqref{e:fixed.pair} for the
pair $(K_1^{(r)}, K_2^{(r)})$ remains feasible in a neighborhood of
$r$. Then
\begin{equation} \label{e:really.limit}
\lim_{u\to\infty} \frac{1}{u^2} \log \Psi_{\cC}(u;  r) = -\frac12
D_{\cC}(r)\,.
\end{equation}
\end{corollary}
\begin{proof}
It follows from Theorem \ref{t:LDP.path} that we only need to show
that
\begin{equation} \label{e:statmnt1}
\lim_{\tau\uparrow r}D_{\cC}(\tau) = D_{\cC}(r)\,.
\end{equation}
However, by the assumption of feasibility,  as $\tau\uparrow r$, 
$$
D_{\cC}(\tau) \leq D_{K_1^{(r)}, K_2^{(r)}(\tau)} \to D_{K_1^{(r)}, K_2^{(r)}}(r) 
= D_{\cC}(r)\,,
$$
giving us the only non-trivial part of \eqref{e:statmnt1}. 
\end{proof}

It turns out that under certain assumptions, given that the event in
\eqref{e:two.sets} occurs,  the random field $u^{-1}\BX$ converges in
law, as $u\to\infty$, to a deterministic function on $T$, ``the most
likely shape of the field''. This is described in the following
result. 
\begin{theorem} \label{e:shape}
Under the conditions of Theorem \ref{t:LDP.path}, suppose that there
is a unique $(K_1^{(r)}, K_2^{(r)})\in \cC$ such that 
\begin{equation} \label{e:unique.pair}
D_{\cC}(r) = D_{K_1^{(r)},K_2^{(r)}}(r)<\infty\,, 
\end{equation}
and such that the optimization problem \eqref{e:fixed.pair} for the
pair $(K_1^{(r)}, K_2^{(r)})$ remains feasible in a neighborhood of
$r$. Then for any $\vep>0$, 
\begin{equation} \label{e:most.likely}
P\biggl( \sup_{\bt\in T} \left| \frac1u X(\bt) -
    x_\cC(\bt)\right|\geq \vep\bigg|  \text{ there is $(K_1,K_2)\in \cC$
    such that} 
\end{equation}
$$
\text{
$X(\bt)>u$ for each $\bt\in K_1$  and $X(\bt)<r u$ for each
    $\bt\in K_2$}\biggr)\to 0
$$
as $u\to\infty$. Here
$$
x_\cC(\bt) = E\left(  X(\bt)H\bigl( K_1^{(r)},K_2^{(r)}\bigr)\right), \, \bt\in T\,,
$$
and $H\bigl( K_1^{(r)},K_2^{(r)}\bigr)$ is the unique minimizer in the
optimization problem \eqref{e:fixed.pair} for the pair $(K_1^{(r)},
K_2^{(r)})$. 
\end{theorem}
\begin{proof}
Using Theorem 3.4.5 in \cite{deuschel:stroock:1989}, we see that
$$
\limsup_{u\to\infty}
\frac{1}{u^2} \log P\biggl( \sup_{\bt\in T} \left| \frac1u X(\bt) -
    x_\cC(\bt)\right|\geq \vep \ \  \text{and there is $(K_1,K_2)\in \cC$
   } 
$$
$$
\text{ such that 
$X(\bt)>u$ for each $\bt\in K_1$  and $X(\bt)< r u$ for each
    $\bt\in K_2$}\biggr)
$$
$$
\leq -\frac12 D_\cC(r; \vep)\,,
$$
where 
\begin{equation}\label{e:opt.eps}
D_{\cC}(r;\vep)
=\inf\biggl\{
EH^2 \: H\in {\mathcal L},\ \sup_{\bt\in T} \left| E(X(\bt)H) -
    x_\cC(\bt)\right|\geq \vep 
\end{equation}
$$
   \text{and for some  $(K_1,K_2)\in \cC$, }
$$
$$
\text{
  $E\bigl(X(\bt)H\bigr)\geq 1$ for each $\bt\in K_1$  and $E\bigl(
X(\bt)H\bigr)\leq r $ for each
    $\bt\in K_2$}  \biggr\}.
$$
Therefore, the claim of the theorem will follow once we prove that 
$D_{\cC}(r;\vep)>D_{\cC}(r)$. Indeed, suppose that the two minimal
values coincide. Let $H_\vep$ be an optimal solution for the problem
\eqref{e:opt.eps}. Since $H\bigl( K_1^{(r)},K_2^{(r)}\bigr)$ is not
feasible for the latter problem, we know that  $H\bigl(
K_1^{(r)},K_2^{(r)}\bigr)\not= H_\vep$, while the two elements have
equal norms. Since $H_\vep$ is feasible for the problem
\eqref{e:optimization.big}, because of the assumed uniqueness of the
pair $(K_1^{(r)}, K_2^{(r)})$ in \eqref{e:unique.pair}, it must also
be feasible for the problem \eqref{e:fixed.pair} with this 
pair $(K_1^{(r)}, K_2^{(r)})$, hence optimal for that problem. This,
however, contradicts the uniqueness property in part (a) of Theorem
\ref{t:dual.repr}. 
\end{proof}

\section{Optimal measures} \label{sec:optimal.measures}

Theorem \ref{t:dual.repr} together with \eqref{e:connect.R} provide a
way to understand the asymptotic behaviour of the probability in
\eqref{e:LD.limit}.  The problem of finding the two minima in the
right hand side of \eqref{e:dual} is not always simple since it is
often unclear how to find the optimal probability measure(s) in these
optimization problems. In this section we provide some results that
are helpful for this task.

We start with the first minimization problem on the right hand side of
\eqref{e:dual}. In this case we can provide necessary and sufficient
condition for a probability measure to be optimal.
\begin{theorem} \label{t:condition.mu.1}
A probability measure $\mu\in M_1^+(K_1)$ is optimal in the
minimization problem
$$
\min_{\mu\in M_1^+(K_1)}
\int_{K_1}\int_{K_1} R_\BX( \bt_1,\bt_2)\, \mu(d\bt_1)\mu(d\bt_2)
$$
if and only if  
$$
\int_{K_1}\int_{K_1} R_\BX( \bt_1,\bt_2)\,
\mu(d\bt_1)\mu(d\bt_2) 
= \min_{\bt_2\in K_1} \int_{K_1} R_\BX( \bt_1,\bt_2)\,
\mu(d\bt_1)\,.
$$
\end{theorem}
This theorem can be proved in the same manner as part (ii) of Theorem
4.3 in \cite{adler:moldavskaya:samorodnitsky:2014}, so we do not
repeat the argument. 
 
Next, observe that if the constraint \eqref{e:cond.1} in the second
minimization problem in \eqref{e:dual} holds with equality, then 
$$
\int_{K_1}\int_{K_1} R_\BX( \bt_1,\bt_2)\, \mu_1(d\bt_1)\mu_1(d\bt_2)
= \frac{A_{K_1,K_2}(\mu_1,\mu_2)}{B_{K_1,K_2}(\mu_1,\mu_2; r)}\,,
$$
so it is of  particular interest  to consider optimality of $\mu_1\in
M_1^+(K_1)$ and $\mu_2\in M_1^+(K_2)$ for the second
minimization problem in \eqref{e:dual} when the inequality in
\eqref{e:cond.1} is strict. It turns out that we can shed some light
on this question in an important special case, when one of the sets
$K_1$ or $K_2$ is a singleton. For the purpose of this discussion we
will assume that the set $K_2$ is a singleton.

Let, therefore, $K_2=\{ \bb\}$, for some $\bb\in\bbr^d$ such that
${\rm Var}(X(\bb))>0$.  In that case  the second optimization problem
in \eqref{e:dual} turns out to be of the form 
\begin{equation} \label{e:problem2}
\min_{\mu\in M_1^+(K_1)} \frac{\int_{K_1}\int_{K_1} R_\BX^{(1)}( \bt_1,\bt_2)\,
\mu(d\bt_1)\mu(d\bt_2)}{\int_{K_1}\int_{K_1} R_\BX^{(2)}( \bt_1,\bt_2)\,
\mu(d\bt_1)\mu(d\bt_2)}
\end{equation}
subject to 
\begin{equation} \label{e:condition.b}
\int_{K_1} R_\BX( \bt,\bb)\, \mu(d\bt)\geq r \int_{K_1} R_\BX( \bt_1,\bt_2)\,
\mu(d\bt_1)\mu(d\bt_2)\,,
\end{equation}
where
$$
R_\BX^{(1)}( \bt_1,\bt_2)= R_\BX( \bt_1,\bt_2) R_\BX(\bb,\bb) -
R_\BX( \bt_1,\bb)  R_\BX( \bt_2,\bb) 
$$
and 
$$
R_\BX^{(2)}( \bt_1,\bt_2)=r^2 R_\BX( \bt_1,\bt_2) -r\bigl( R_\BX(
\bt_1,\bb) +  R_\BX( \bt_2,\bb) \bigr) + R_\BX(\bb,\bb)\,.
$$
Notice that both $R_\BX^{(1)}$ and $R_\BX^{(2)}$ are nonnegative
definite, i.e. legitimate covariance functions on $T$. In fact, up to
the positive factor $ R_\BX(\bb,\bb) $, the function $R_\BX^{(1)}$ is
the conditional covariance function of the random field $\BX$ given
$X(\bb)$, while $R_\BX^{(2)}$ is the covariance function of the random
field 
$$
Y(\bt) = rX(\bt) - X(\bb), \, \bt\in T\,.
$$
This problem is a generalization of the first
optimization problem in \eqref{e:dual}, with the optimization of a single
integral of a covariance function replaced by the  optimization of a
ratio of the integrals of two covariance functions. 
 
The following result presents necessary conditions for optimality  in
the optimization problem \eqref{e:problem2}  of a measure for which
the constraint \eqref{e:condition.b} is satisfied as a strict
inequality. Note that the validity of the theorem does not depend on 
particular forms for $R_\BX^{(1)}$ and $R_\BX^{(2)}$. Observe that the
nonnegative definiteness of $R_\BX^{(1)}$ and $R_\BX^{(2)}$ means that
both the numerator and the denominator in \eqref{e:problem2} are
nonnegative. If the denominator vanishes at an optimal measure, then
the numerator must vanish as well (and the ratio is then determined via a
limiting procedure). In the theorem we assume that the denominator
does not vanish. 
\begin{theorem} \label{t:nec.cond2}
Let $\mu\in M_1^+(K_1)$ be such that \eqref{e:condition.b} holds as a strict
inequality. Let $\mu$ be optimal in
the optimization problem \eqref{e:problem2} and 
$$
\int_{K_1}\int_{K_1} R_\BX^{(2)}( \bt_1,\bt_2)\,
\mu(d\bt_1)\mu(d\bt_2)>0\,. 
$$
Then 
\begin{equation} \label{e:nec.cond}
 \int_{K_1} R_\BX^{(1)}( \bt_1, \bt)\,
\mu(d\bt_1) \int_{K_1}\int_{K_1} R_\BX^{(2)}( \bt_1, \bt_2)\,
\mu(d\bt_1) \mu(d\bt_2)
\end{equation}
$$
\geq \int_{K_1} R_\BX^{(2)}( \bt_1, \bt)\,
\mu(d\bt_1) \int_{K_1}\int_{K_1} R_\BX^{(1)}( \bt_1, \bt_2)\,
\mu(d\bt_1) \mu(d\bt_2)
 $$
for every $\bt\in K_1$. Moreover, \eqref{e:nec.cond} holds as as equality 
$\mu$-almost everywhere. 
\end{theorem}
\begin{proof}
Let
$$
\Psi(\eta) = \frac{\int_{K_1}\int_{K_1} R_\BX^{(1)}( \bt_1,\bt_2)\,
\eta(d\bt_1)\eta(d\bt_2)}{\int_{K_1}\int_{K_1} R_\BX^{(2)}( \bt_1,\bt_2)\,
\eta(d\bt_1)\eta(d\bt_2)}
$$
for those $\eta\in M(K_1)$, the space of finite signed measures on
$K_1$ for which the denominator does not vanish. It is elementary to
check that $\Psi$ is Fr\'echet differentiable at every such point, in
particular at the optimal $\mu$ in the theorem. Its Fr\'echet
derivative at $\mu$ is given by
$$
D\Psi(\mu)[\eta] = \frac{2}{\left(\int_{K_1}\int_{K_1} R_\BX^{(2)}( \bt_1,\bt_2)\,
\mu(d\bt_1)\mu(d\bt_2)\right)^2}
$$
$$
\left( \int_{K_1}\int_{K_1} R_\BX^{(2)}( \bt_1,\bt_2)\,
\mu(d\bt_1)\mu(d\bt_2) \int_{K_1}\int_{K_1} R_\BX^{(1)}( \bt_1,\bt_2)\,
\mu(d\bt_1)\eta(d\bt_2) \right.
$$
$$
\left. 
- \int_{K_1}\int_{K_1} R_\BX^{(1)}( \bt_1,\bt_2)\,
\mu(d\bt_1)\mu(d\bt_2) \int_{K_1}\int_{K_1} R_\BX^{(2)}( \bt_1,\bt_2)\,
\mu(d\bt_1)\eta(d\bt_2)\right)
$$
for $\eta\in M(K_1)$. We view the problem \eqref{e:problem2} as the
minimization problem (2.1) in \cite{molchanov:zuyev:2004}. In our case
the set $A$ coincides with the cone $M_1^+(K_1)$ of probability measures,
the set $C$ is the negative half-line $(-\infty,0]$, and $H:\,
M(K_1)\to\bbr$ is given by 
$$
H(\eta) = r \int_{K_1}\int_{K_1} R_\BX( \bt_1,\bt_2)\,
\eta(d\bt_1)\eta(d\bt_2) - \int_{K_1} R_\BX( \bt,\bb)\, \eta(d\bt)\,.
$$
This function is also easily seen to be Fr\'echet differentiable at $\mu$, and
$$
DH(\mu)[\eta]= 2r \int_{K_1}\int_{K_1} R_\BX( \bt_1,\bt_2)\,
\mu(d\bt_1)\eta(d\bt_2) - \int_{K_1} R_\BX( \bt,\bb)\, \eta(d\bt)
$$
for $\eta\in M(K_1)$. Finally, the fact that \eqref{e:condition.b} holds as a strict
inequality implies that the measure $\mu$ is regular according to
Definition 2.1 in \cite{molchanov:zuyev:2004}. 

The claim \eqref{e:nec.cond} now follows from Theorem 3.1 in \cite{molchanov:zuyev:2004}. 
\end{proof}

If, for example, the covariance function $R^{(2)}_\BX$ is strictly
positive on $K_1$, then an alternative way of writing the conclusion
of Theorem \ref{t:nec.cond2} is 
$$
\frac{\int_{K_1} R_\BX^{(1)}( \bt_1, \bt)\,
\mu(d\bt_1)}{\int_{K_1} R_\BX^{(2)}( \bt_1, \bt)\,
\mu(d\bt_1)}
\geq \frac{\int_{K_1}\int_{K_1} R_\BX^{(1)}( \bt_1, \bt_2)\,
\mu(d\bt_1) \mu(d\bt_2)}{\int_{K_1}\int_{K_1} R_\BX^{(2)}( \bt_1, \bt_2)\,
\mu(d\bt_1) \mu(d\bt_2)}
$$
for every $\bt\in K_1$, with equality for $\mu$-almost every
$\bt$. This is a condition of the same nature as the condition in
Theorem \ref{t:condition.mu.1}. The convexity of the double integral
as a function of the measure $\mu$ in the optimization problem in
Theorem \ref{t:condition.mu.1}  makes the necessary condition for
optimality also sufficient. This convexity is lost in Theorem
\ref{t:nec.cond2}, and it is not clear at the moment when the
necessary condition in that theorem is also sufficient. 

We conclude this section with an explicit computation of the limiting
shape $x_\cC$ in Theorem \ref{e:shape} in terms of the optimal
measures in the dual problem. We restrict ourselves to the case  where
the optimal pair $(K_1^{(r)}, K_2^{(r)})$ is such that $K_2^{(r)}$ is a
singleton. This would always be the case, of course, if we considered a
family $\cC$ consisting of a single pair of sets, $(K_1,K_2)$, with
$K_2$ a singleton, to start with. 
\begin{theorem} \label{t:shape.comp}
Under the conditions of Theorem \ref{e:shape}, assume that the set
$K_2^{(r)}=\{\bb\}$ is a singleton. Let $\mu^{(r)}\in M_1^+(K_1)$ be
the optimal measure in the optimization problem \eqref{e:dual} for
the pair $(K_1^{(r)}, K_2^{(r)})$. Then 
\begin{equation} \label{e:shape.comp1}
x_\cC(\bt) = D_\cC(r) \int_{K_1} R_\BX(\bt,\bt_1)\, \mu^{(r)}(d\bt_1),
\, \bt\in T\,,
\end{equation}
if the first minimum in \eqref{e:dual} does not exceed the second
minimum, and
\begin{equation} \label{e:shape.comp2}
x_\cC(\bt) = a\bigl( \mu^{(r)}\bigr)\left[ \int_{K_1}
  R_\BX(\bt,\bt_1)\, \mu^{(r)}(d\bt_1) - b\bigl( \mu^{(r)}\bigr)
  R_\BX(\bt,\bb)\right], \, \bt\in T\,,
\end{equation}
if the first minimum in \eqref{e:dual} is larger than the second
minimum. Here 
\begin{equation} \label{e:a.term}
a\bigl( \mu^{(r)}\bigr) = 
\end{equation}
$$
\frac{R_\BX(\bb,\bb) - r \int_{K_1} R_\BX(
  \bt_1,\bb)\,  \mu^{(r)}(d\bt_1)}{R_\BX(\bb,\bb) \int_{K_1}\int_{K_1} R_\BX(
  \bt_1,\bt_2)\,  \mu^{(r)}(d\bt_1) \mu^{(r)}(d\bt_2) - \left( \int_{K_1} R_\BX(
  \bt_1,\bb)\,  \mu^{(r)}(d\bt_1)\right)^2}
$$
and 
\begin{equation} \label{e:b.term}
b\bigl( \mu^{(r)}\bigr) = 
\end{equation}
$$
\frac{r \int_{K_1}\int_{K_1} R_\BX(
  \bt_1,\bt_2)\,  \mu^{(r)}(d\bt_1) \mu^{(r)}(d\bt_2) -  \int_{K_1} R_\BX(
  \bt_1,\bb)\,  \mu^{(r)}(d\bt_1)}{r \int_{K_1} R_\BX(
  \bt_1,\bb)\,  \mu^{(r)}(d\bt_1) - R_\BX(\bb,\bb)}\,.
$$
\end{theorem}
\begin{remark} \label{rk:shape}
{\rm
Notice that, since the set $K_2=\{\bb\}$ is  a singleton, only
a measure in $M_1^+(K_1)$ is a variable over which one can optimize,
as $M_1^+(K_2)$ consists of a single measure, the point mass at
$\bb$. Notice also that we are using the same name, $\mu^{(r)}$, for
the optimal measure throughout Theorem \ref{t:shape.comp} for
notational convenience only, because in the two different cases
considered in the theorem, it referes to optimal solutions to two
different problems. 
}
\end{remark}
\begin{proof}[Proof of Theorem \ref{t:shape.comp}]
By Theorem  \ref{e:shape} all we need to do is to prove the following
representations of the unique minimizer $H\bigl(
K_1^{(r)},K_2^{(r)}\bigr)$  in the
optimization problem \eqref{e:fixed.pair} for the pair $(K_1^{(r)},
K_2^{(r)})$. If the first minimum in \eqref{e:dual} does not exceed the second
minimum, then 
\begin{equation} \label{e:min.1}
H\bigl( K_1^{(r)},K_2^{(r)}\bigr) = D_\cC(r) \int_{K_1}  X(\bt_1)\,
\mu^{(r)}(d\bt_1)\,, 
\end{equation}
and, if the first minimum in \eqref{e:dual} is larger than the second
minimum, then 
\begin{equation} \label{e:min.2}
H\bigl( K_1^{(r)},K_2^{(r)}\bigr) =a\bigl( \mu^{(r)}\bigr)\left[ \int_{K_1}
 X(\bt_1)\, \mu^{(r)}(d\bt_1) - b\bigl( \mu^{(r)}\bigr)
  X(\bb)\right]\,.
\end{equation}

We start by observing that, under the assumptions of Theorem
\ref{e:shape},  the feasible set in the 
optimization problem \eqref{e:fixed.pair} for the pair $(K_1^{(r)},
K_2^{(r)})$ has an interior point. Therefore, Theorem 1 (p. 224) in
\cite{luenberger:1969} applies. It follows that the vector $H\bigl(
K_1^{(r)},K_2^{(r)}\bigr)$ solves the inner minimization problem in 
\eqref{e:optimize.lagrange} when we use
$$
\mu_1= m_1 \mu^{(r)}, \ \ \mu_2= m_2\delta_\bb\,,
$$ 
where $m_1$ and $m_2$ are nonnegative numbers solving the optimization
problem \eqref{e:sep.prob} corresponding to the measures $\mu^{(r)}$
and $\delta_\bb$. It follows immediately that $H\bigl(
K_1^{(r)},K_2^{(r)}\bigr)$ must be of the form
\begin{equation} \label{e:H.interm}
H\bigl( K_1^{(r)},K_2^{(r)}\bigr) =a \left[ m_1\int_{K_1}
 X(\bt_1)\, \mu^{(r)}(d\bt_1) -m_2  X(\bb)\right]
\end{equation}
for some $a\geq 0$. 

We now consider separately the two cases of the theorem. Suppose first
that the first minimum in \eqref{e:dual} does not exceed the second
minimum. In that case we have $m_2=0$ above, see Remark
\ref{rk:m1m2}. According to that remark, this happens when 
\begin{equation} \label{e:case1}
\int_{K_1}R_\BX(\bt_1,\bb)\, \mu^{(r)}(d\bt_1)\leq
r \int_{K_1}\int_{K_1}R_\BX(\bt_1,\bt_2)\, \mu^{(r)}(d\bt_1)
\mu^{(r)}(d\bt_2)\,.
\end{equation}
We combine, in this case, $a$ and $m_1$ in
\eqref{e:H.interm} into a single nonnegative constant, which we still
denote by $a$. We then consider vectors of the form
\begin{equation} \label{e:cand.H1}
H\bigl( K_1^{(r)},K_2^{(r)}\bigr) =a \int_{K_1}
 X(\bt_1)\, \mu^{(r)}(d\bt_1)
\end{equation}
as candidates for the optimal solution in \eqref{e:fixed.pair}.  The
statement \eqref{e:min.1} will follow once we show 
that $a= D_\cC(r)$ is the optimal value of $a$. By Theorem
\ref{t:dual.repr}, we need to show that the optimal value of $a$ is
\begin{equation} \label{e:cand.a1}
a= \left(  \int_{K_1}\int_{K_1} R_\BX( \bt_1,\bt_2)\,
  \mu^{(r)}(d\bt_1)\mu^{(r)}(d\bt_2), \ \right)^{-1}\,.
\end{equation}
The first step is to check that using $a$ given by \eqref{e:cand.a1}
in \eqref{e:cand.H1} leads to a feasible solution to the problem
\eqref{e:fixed.pair}.  Indeed, the fact that the constraints of the
type ``$\geq$'' in that problem are satisfied follows from the
optimality of the measure $\mu^{(r)}$ and Theorem
\ref{t:condition.mu.1}. The fact that the constraint of the
type ``$\leq$'' in that problem is satisfied follows from
\eqref{e:case1}. This establishes the feasibility of the solution. Its
optimality now follows from the fact that using $a$ given by \eqref{e:cand.a1}
in \eqref{e:cand.H1} leads to a feasible solution whose second moment
is equal to the optimal value  $D_\cC(r)$. 

Suppose now that the first minimum in \eqref{e:dual} is larger than
the second minimum. According to Remark \ref{rk:m1m2} this happens
when \eqref{e:case1} fails and, further, we have
$$
\frac{m_1}{m_2} = \left( b\bigl( \mu^{(r)}\bigr)\right)^{-1}\,,
$$
where $b\bigl( \mu^{(r)}\bigr)$ is defined in \eqref{e:b.term}. 
Combining, once again, $a$ and $m_1$ in \eqref{e:H.interm} into a
single nonnegative constant, which is still denoted by $a$, we 
consider vectors of the form
\begin{equation} \label{e:cand.H2}
H\bigl( K_1^{(r)},K_2^{(r)}\bigr) =a\left[  \int_{K_1}
 X(\bt_1)\, \mu^{(r)}(d\bt_1) - b\bigl( \mu^{(r)}\bigr) X(\bb)\right]
\end{equation}
as candidates for the optimal solution in \eqref{e:fixed.pair}.  The
proof will be complete once we show that the value of $a=a\bigl(
\mu^{(r)}\bigr)$ given in \eqref{e:a.term} is the optimal value of
$a$. 

Notice that for vectors of the form \eqref{e:cand.H2}, the optimal
value of $a$ solves the optimization problem 
$$
\min_{a\geq 0} \ \ a
$$
\begin{equation} \label{e:last.prob}
\hskip 0.2in \text{subject to}
\end{equation}
$$
a\left[ \int_{K_1} R_\BX(\bt,\bt_1)\, \mu^{(r)}(d\bt_1) -
  b\bigl(\mu^{(r)}\bigr) R_\BX(\bt,\bb)\right] \geq 1 \ \ \text{for
  each $\bt\in K_1$}\,,
$$
$$
a\left[ \int_{K_1} R_\BX(\bb,\bt_1)\, \mu^{(r)}(d\bt_1) -
  b\bigl(\mu^{(r)}\bigr) R_\BX(\bb,\bb)\right] \leq r\,.
$$
The first step is to check that the value of $a=a\bigl(
\mu^{(r)}\bigr)$ given in \eqref{e:a.term} is feasible for the problem
\eqref{e:last.prob}. First of all, nonnegativity of this value of $a$
follows from the fact that \eqref{e:case1} fails. Furthermore, it
takes only simple algebra to check that the 
``$\leq$'' constraint is satisfied as an equality. In order to see that 
  the ``$\geq$'' constraints are satisfied as well, notice that, since
  \eqref{e:case1} fails, we are in the situation of Theorem
  \ref{t:nec.cond2}. Therefore, the measure $\mu^{(r)}$ satisfies the
  necessary conditions for optimality given in
  \eqref{e:nec.cond}. Again, it 
takes only elementary algebraic calculations to see that these optimality
conditions are equivalent to the ``$\geq$'' constraints in the problem
\eqref{e:last.prob}. 

Now that the feasibility has been established, the optimality of the
solution to the problem \eqref{e:fixed.pair} given by using in
\eqref{e:cand.H2} the value
of $a=a\bigl( \mu^{(r)}\bigr)$ from \eqref{e:a.term}, follows, once again, from the fact that  this feasible solution has second moment
 equal to the optimal value  $D_\cC(r)$, as can be checked by easy
algebra. 
\end{proof}

\section{Isotropic random fields} \label{sec:isotropic}

In this section we will consider stationary isotropic Gaussian random
fields, i.e. random fields for which 
$$
R_\BX( \bt_1, \bt_2)= R(\| \bt_1-\bt_2\|), \, \bt_1,\, \bt_2\in T\,,
$$
for some function $R$ on $[0,\infty)$. We will concentrate on the
asymptotic behaviour of the probabilities $\Psi_{\text sp}(u; r)$
and $\Psi_{\text sp; c}(u; r)$ in \eqref{e:spheres} and
\eqref{e:spheres.center} correspondingly. 

We consider the probability $\Psi_{\text sp; c}(u; r)$ first. In
this case, by \eqref{e:connect.R} and isotropy, 
\begin{equation} \label{e:connect.center}
D_{\cC}(r) = \min_{0\leq \rho\leq D} M_\rho(r)\,,
\end{equation}
where 
\begin{equation} \label{e:diam}
D= \sup\bigl\{ \rho\geq 0:\ \text{there is a ball of radius $\rho$
  entirely in $T$}\bigr\}\,,
\end{equation}
and $M_\rho(r)= D_{K_1,K_2}(r) $ in \eqref{e:fixed.pair} with $K_1$
being the sphere of radius $\rho$ centered at the origin, and
$K_2=\{\bnull\}$. The following result provides a fairly detailed
description of the asymptotic behaviour of the probability
$\Psi_{\text sp; c}(u; r)$. 
\begin{theorem} \label{t:isotropic.centers}
Let $\BX$ be isotropic. Then
\begin{equation} \label{e:centers.exact}
\lim_{u\to\infty} \frac{1}{u^2} \log \Psi_{\text sp; c}(u; r) =
-\frac12 \min_{0\leq \rho\leq D} M_\rho(r)\,.
\end{equation}
Furthermore, for every $0<r\leq 1$,  
$M_\rho(r) = \bigl( W_\rho(r)\bigr)^{-1}$, where 
\begin{equation} \label{e:M.r}
W_\rho(r) = \left\{ 
\begin{array}{ll}
D(\rho) & \text{if} \ R(\rho)\leq rD(\rho)\,, \\
\frac{R(0)D(\rho)-\bigl( R(\rho)\bigr)^2}{R(0)-2rR(\rho) + r^2D(\rho)}
& \text{if} \ R(\rho)> rD(\rho)\,.
\end{array}
\right.
\end{equation}
Here 
\begin{equation} \label{e:def.D}
D(\rho) = \int_{S_\rho(\bnull)} \int_{S_\rho(\bnull)}
R(\|\bt_1-\bt_2\|)\mu_h(d\bt_1)\, \mu_h(d\bt_2)\,,
\end{equation}
where $S_\rho(\bnull)$ is the sphere of radius $\rho$ centered at the
origin, and $\mu_h$ is the rotation invariant probability measure on
that sphere. 
\end{theorem}
\begin{proof}
We use part (b) of Theorem \ref{t:dual.repr} with $K_1=S_\rho(\bnull)$
and $K_2=\{\bnull\}$. Note first of all that by the rotation
invariance of the measure $\mu_h$, the function
$$
\int_{K_1} R(\| \bt_1-\bt_2\|)\, \mu_h(d\bt_1), \ \bt_2\in
S_\rho(\bnull)\,,
$$
is  constant. Hence by Theorem \ref{t:condition.mu.1}
the measure $\mu_h$ is optimal in the first minimization problem on the
right hand side of \eqref{e:dual}, and the optimal value in that
problem is $D(\rho)$. 

In the second minimization problem on the
right hand side of \eqref{e:dual}, since $K_2$ is a singleton, the
optimization is only over measures $\mu_1\in M_1^+(K_1)$, and so we
drop the unnecessary $\mu_2$ in the argument in the ratio in that
problem. By the isotropy of the field,
$$
\frac{A_{K_1,K_2}(\mu_1)}{B_{K_1,K_2}(\mu_1; r)} = 
\frac{\int_{K_1}\int_{K_1}  R(\| \bt_1-\bt_2\|)\,
  \mu_1(d\bt_1)\mu_1(d\bt_2) R(0)- \bigl(R(\rho)\bigr)^2}{r^2
  \int_{K_1}\int_{K_1}  R(\| \bt_1-\bt_2\|)\,
  \mu_1(d\bt_1)\mu_1(d\bt_2) -2r R(\rho)+ R(0)}
$$
$$
= \frac{R(0)}{r^2} - \frac{(R(\rho)-R(0)/r)^2}{r^2
  \int_{K_1}\int_{K_1}  R(\| \bt_1-\bt_2\|)\,
  \mu_1(d\bt_1)\mu_1(d\bt_2) -2r R(\rho)+ R(0)}\,.
$$
Since the expression in the denominator is nonnegative (see the
discussion following \eqref{e:condition.b}), the ratio in the left
hand side is 
smaller if the double integral in the right hand side is
smaller. Furthermore,  condition \eqref{e:cond.1} reads, in this
case, as
$$
R(0)\geq r \int_{K_1}\int_{K_1}  R(\| \bt_1-\bt_2\|)\,
  \mu_1(d\bt_1)\mu_1(d\bt_2)\,.
$$
This means that, if this condition is satisfied when the double
integral is large, it is also satisfied when the double
integral is small. Recalling that the double integral is smallest
when $\mu=\mu_h$, we conclude that 
$$
\min_{\mu_1\in M_1^+(K_1) \atop 
\text{\rm subject to \eqref{e:cond.1}}}
\frac{A_{K_1,K_2}(\mu_1)}{B_{K_1,K_2}(\mu_1; r)} 
= \left\{ \begin{array}{ll}
\infty & \text{if $R(\rho)< rD(\rho)$,} \\
\frac{R(0)D(\rho) -\bigl( R(\rho)\bigr)^2}{R(0)-2rR(\rho) +
  r^2D(\rho)} & \text{if $R(\rho)\geq rD(\rho)$. } 
\end{array} \right. 
$$
Finally, since
\begin{equation} \label{e:H.D}
\frac{R(0)D(\rho) -\bigl( R(\rho)\bigr)^2}{R(0)-2rR(\rho) +
  r^2D(\rho)} 
\end{equation}
$$
=
D(\rho) - \frac{\bigl( rD(\rho)-R(\rho)\bigr)^2}{R(0)-2rR(\rho) +
  r^2D(\rho)} \leq D(\rho)\,,
$$
we obtain \eqref{e:M.r}. 

It remains to prove \eqref{e:centers.exact}. We use \eqref{e:connect.center}. By 
Theorem
\ref{t:LDP.path},  it is enough to prove that 
the function $D_{\cC}$ is left continuous. By monotonicity, if
$D_{\cC}=\infty$ for some $r>0$, then the same is true for all smaller
values of the argument, and the left continuity is trivial. Let,
therefore, $0<r\leq 1$ be such that $D_{\cC}<\infty$. 
Let $0\leq \rho_0\leq D$ be such that 
$$
M_{\rho_0}(r) = \min_{0\leq \rho\leq D} M_\rho(r)\,.
$$
Then $W_{\rho_0}(r)>0$. By \eqref{e:M.r}, the $W_\rho(r)$ is, for a
fixed $\rho$, a continuous function of $r$. Therefore,
$$
\lim_{s\uparrow r}  D_{\cC}(s) \leq \lim_{s\uparrow r} \bigl(
W_\rho(s)\bigr)^{-1} = 
\bigl(
W_\rho(r)\bigr)^{-1} = D_{\cC}(r)\,.
$$
By the monotonicity of the function $D_{\cC}$, this implies left
continuity. 
\end{proof}

The distinction between the situations described by the two conditions
on the right hand side of \eqref{e:M.r} can be described using the
intuition introduced in the discussion following Example
\ref{rk:cty}. If there is a ``peak'' of height greater than $u$ covering the
entire sphere of radius $\rho$, is it likely that there will be a
``hole'' in the center of the sphere where the height is smaller than
$ru$? Theorem \ref{t:isotropic.centers} says that a hole is likely if
$R(\rho)\leq rD(\rho)$ and unlikely if $R(\rho)> rD(\rho)$, at least
at the logarithmic level. 

It is reasonable to expect that, for spheres of a very small radius, a
hole in the center is unlikely, while for spheres of a very large radius, a
hole in the center is likely, at least if the terms ``very
small'' and ``very large'' are used relatively to the depth of the
hole described by the factor $r$. This intuition turns out to be
correct in many, but not all, cases, and some unexpected phenomena
emerge. We will try to clarify the situation in the subsequent
discussion.

We look at spheres of  very small radius first. Observe first
that by the continuity of the covariance function, we have both
$R(\rho)\to R(0)$ and $D(\rho)\to R(0)$ as $\rho\to 0$. Therefore, if
$0<\rho<1$, then the condition $R(\rho)> rD(\rho)$ holds for spheres
of sufficiently small radii, and a hole that deep is, indeed,
unlikely. Is the same true for $r=1$? In other words, is it true that
there is $\delta>0$ such that
\begin{equation} \label{e:small.radius}
D(\rho)<R(\rho) \ \ \text{for all $0<\rho<\delta$?}
\end{equation}
A sufficient condition is that the function $R$ is concave on
$[0,2\delta]$; this is always the case for a sufficiently small
$\delta>0$ if the covariance function $R$ corresponds to a spectral
measure with a finite second moment. To see how the concavity implies
\eqref{e:small.radius}, note that by the Jensen inequality,
$$
D(\rho) \leq R\left( \int_{S_\rho(\bnull)} \int_{S_\rho(\bnull)}
 \|\bt_1-\bt_2\|\mu_h(d\bt_1)\, \mu_h(d\bt_2)\right)\,.
$$
Further, by the symmetry of the measure $\mu_h$ and the triangle
inequality, 
$$
\int_{S_\rho(\bnull)} \int_{S_\rho(\bnull)}
 \|\bt_1-\bt_2\|\, \mu_h(d\bt_1)\, \mu_h(d\bt_2)
$$
$$
= \int_{S_\rho(\bnull)} \int_{S_\rho(\bnull)}
 \bigl( \|\bt_1-\bt_2\|/2+ \|\bt_1+\bt_2\|/2\bigr)
\, \mu_h(d\bt_1)\, \mu_h(d\bt_2)
$$
$$
\geq  \int_{S_\rho(\bnull)} \int_{S_\rho(\bnull)} 
\|\bt_1\|\, \mu_h(d\bt_1)\, \mu_h(d\bt_2) = \delta\,.
$$
Since the concavity of $R$ on $[0,2\delta]$ implies its monotonicity,
we obtain \eqref{e:small.radius}. 

In dimensions $d\geq 2$, the hole in the center with $r=1$ may be 
unlikely for small spheres even without concavity.  Consider
covariance functions satisfying 
\begin{equation} \label{e:regvar.zero}
R(\rho) = R(0) -a\rho^\beta + o(\rho^\beta), \  \ \text{as $\rho\to 0$,}
\end{equation}
for some $a>0$ and $1\leq\beta\leq 2$. To see that this implies
\eqref{e:small.radius} as well, notice that, under
\eqref{e:regvar.zero}, 
$$
D(\rho) = R(0) - a \rho^\beta \int_{S_1(\bnull)} \int_{S_1(\bnull)}
 \|\bt_1-\bt_2\|^\beta\, \mu_h(d\bt_1)\, \mu_h(d\bt_2) + o(\rho^\beta),\  \
 \text{as $\rho\to 0$.}
$$
Using, as above, the symmetry together with the Jensen inequality and
the triangle inequality we see that
$$
\int_{S_1(\bnull)} \int_{S_1(\bnull)}
 \|\bt_1-\bt_2\|^\beta\, \mu_h(d\bt_1)\, \mu_h(d\bt_2)
$$
$$
= \int_{S_1(\bnull)} \int_{S_1(\bnull)}
\bigl(  \|\bt_1-\bt_2\|^\beta/2 +  \|\bt_1+\bt_2\|^\beta/2\bigr)\,
\mu_h(d\bt_1)\, \mu_h(d\bt_2) 
$$
$$
\geq \int_{S_1(\bnull)} \int_{S_1(\bnull)}
\bigl(  \|\bt_1-\bt_2\|/2 +  \|\bt_1+\bt_2\|/2\bigr)^\beta\,
\mu_h(d\bt_1)\, \mu_h(d\bt_2) 
$$
$$
> \int_{S_1(\bnull)} \int_{S_1(\bnull)}
 \|\bt_1\|^\beta\, \mu_h(d\bt_1)\, \mu_h(d\bt_2) =1\,.
$$
Thus we see that for some $a_1>a$,
$$
D(\rho) = R(0) -a_1\rho^\beta + o(\rho^\beta), \  \ \text{as $\rho\to 0$,}
$$
and so \eqref{e:small.radius} holds for $\delta>0$ small enough. 

An example of the situation where \eqref{e:regvar.zero} holds without
concavity condition is that of the isotropic Ornstein-Uhlenbeck random
field corresponding to $R(t) = \exp\{ -a|t|\}$ for some $a>0$. It is
interesting that for this random field a hole in the center with $r=1$
is unlikely for small spheres in dimension $d\geq 2$, but not in
 dimension $d=1$. Indeed, in the latter case we have 
$$
D(\rho) = \bigl( 1+ e^{-2a\rho})/2> e^{-a\rho} = R(\rho)\,,
$$
no matter how small $\rho>0$ is.

When $\rho\to \infty$, we expect that a hole in the center of a sphere
will become likely no matter what $0<r\leq 1$ is. According to the
discussion above, this happens when
\begin{equation} \label{e:D.large}
\lim_{\rho\to\infty} \frac{D(\rho)}{R(\rho)}= \infty\,.
\end{equation}
This turns out to be true under certain short memory
assumptions. Assume, for example, that $R$ is nonnegative and 
\begin{equation} \label{e:short.mem}
\liminf_{v\to\infty} \frac{R(tv)}{R(v)} \geq t^{-a} \ \ \text{with $a\geq
  d-1$, for all   $0<t\leq 1$.}
\end{equation}
Then by Fatou's lemma,
\beqq
&&\liminf_{\rho\to\infty}\frac{D(\rho)}{R(\rho)} \\
&&\geq 
\int_{S_1(\bnull)} \int_{S_1(\bnull)} \one\bigl( \|\bt_1-\bt_2\|\leq
1\bigr) \liminf_{\rho\to\infty} \frac{R\bigl(
  \|\bt_1-\bt_2\|\rho\bigr)}{R(\rho)}  \, \mu_h(d\bt_1)\,
\mu_h(d\bt_2) \\
&&\geq 
\int_{S_1(\bnull)} \int_{S_1(\bnull)} \one\bigl( \|\bt_1-\bt_2\|\leq
1\bigr)\,  \|\bt_1-\bt_2\|^{-a} \, \mu_h(d\bt_1)\,
\mu_h(d\bt_2)   = \infty\,,
\eeqq
so that \eqref{e:D.large} holds. 

However, in dimensions $d\geq 2$, the situation turns out to be
different under an assumption of a longer memory. Assume, for
simplicity, that $R$ is monotone, and suppose that, for some
$\vep>0$,
\begin{equation} \label{e:long.mem}
R\ \ \text{is regularly varying at infinity with exponent
  $-(d-1)+\vep$.}
\end{equation}
We claim that, in this case
\begin{equation} \label{e:huge.sphere}
\lim_{\rho\to\infty}  \frac{D(\rho)}{R(\rho)}= \int_{S_1(\bnull)}
\int_{S_1(\bnull)}   \|\bt_1-\bt_2\|^{-(d-1)+\vep} \, \mu_h(d\bt_1)\,
\mu_h(d\bt_2) <\infty\,.
\end{equation}
It is easy to prove this using Breiman's theorem as in, for instance,
Proposition 7.5 in \cite{resnick:2007}. Let $Z$ be a positive random
variable such that $P(Z>z) = R(z)/R(0)$, and let $Y$ be an independent
of $Z$ positive random variable whose law is given by the image of the
product measure $\mu_h\times \mu_h$ on $S_1(\bnull)\times S_1(\bnull)$
under the map $(\bt_1,\bt_2)\mapsto \|\bt_1-\bt_2\|^{-1}$. Notice that
$EY^{d-1-\vep/2}<\infty$. Therefore, by Breiman's theorem, as
$\rho\to\infty$, 
$$
D(\rho) = R(0) P(ZY>\rho) \sim R(0) EY^{d-1-\vep}P(Z>\rho)
$$
$$
= R(\rho) \int_{S_1(\bnull)}  \int_{S_1(\bnull)}
\|\bt_1-\bt_2\|^{-(d-1)+\vep} \, \mu_h(d\bt_1)\, 
\mu_h(d\bt_2)\,.
$$

If we call
$$
I(d;\vep) = \int_{S_1(\bnull)}  \int_{S_1(\bnull)}
\|\bt_1-\bt_2\|^{-(d-1)+\vep} \, \mu_h(d\bt_1)\, 
\mu_h(d\bt_2)\,, 
$$
then we have just proved that the hole in the center of a sphere
corresponding to a factor $r<1/I(d;\vep)$ remains unlikely even for
spheres of infinite radius! This is in spite of the fact, that the 
random field is ergodic, and even mixing, as the covariance function
vanishes at infinity. This phenomenon is impossible if $d=1$ since in
this case $D(\rho)$ does not converge to zero as $\rho\to\infty$. 

Some estimates of the integral $I(d;\vep)$ for $d=2$ and $d=3$ are
presented on Fig. \ref{fig:integral}. 

\begin{figure}[!ht] 
\includegraphics[width=10cm,height=4cm]{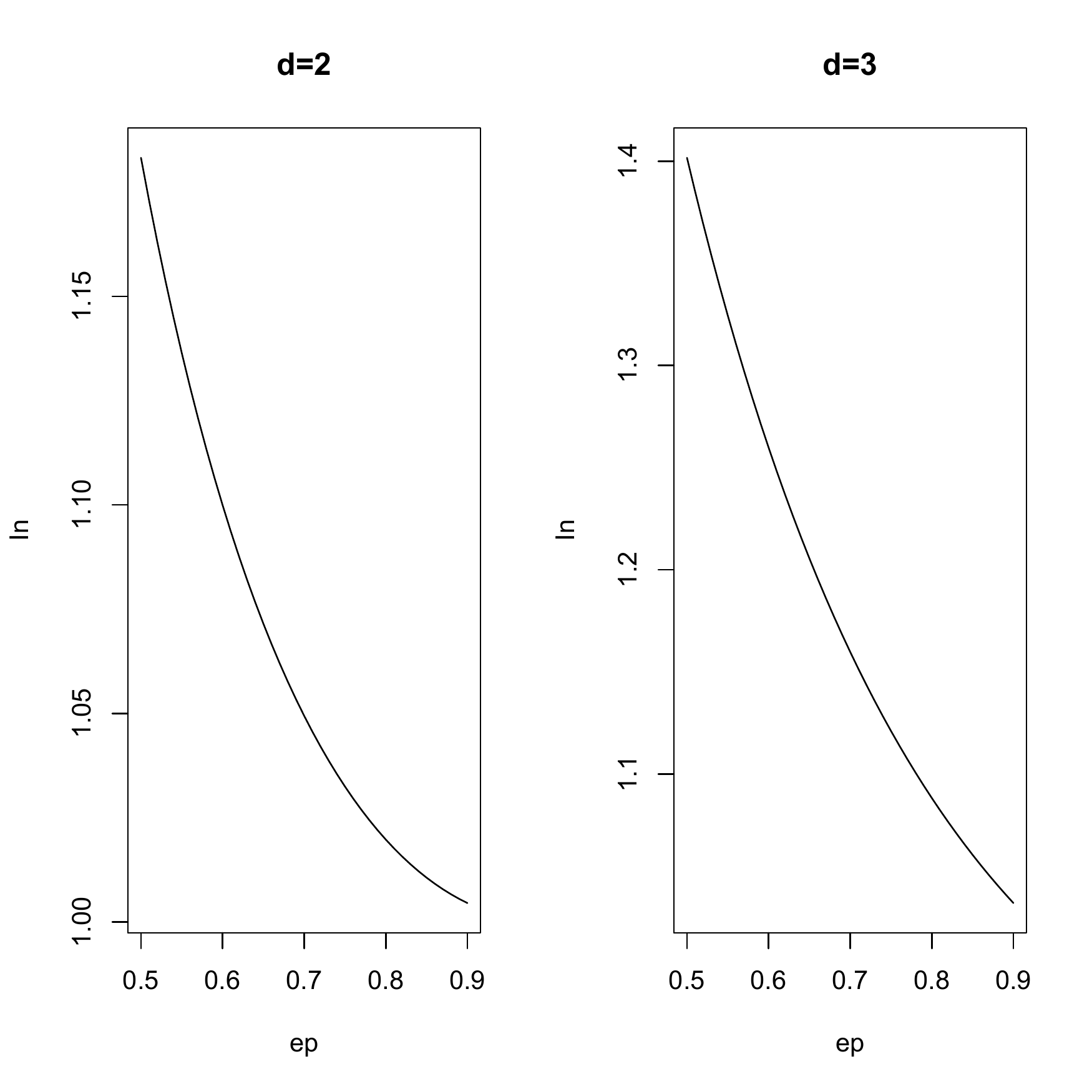}
\caption{The integral $I(d;\vep)$ for $d=2$ and $d=3$}
\label{fig:integral}
\end{figure}

One can pursue the analysis of holes in the center of a sphere a bit
further, and talk about {\it the most likely radius of a sphere} for
which the random field has a ``peak'' of height greater than $u$ covering the
entire sphere, and a ``hole'' in the center of the sphere where the
height is smaller than $ru$, as $u\to\infty$. According to Theorem
\ref{t:isotropic.centers}, this most likely radius is given by ${\rm
  argmax}_{\rho\geq 0} W_\rho(r)$. The following corollary shows how
calculate this most likely radius. For simplicity, we assume that $R$
is monotone and $0<r<1$. Let
$$
H_\rho(r) = \frac{R(0)D(\rho)-\bigl( R(\rho)\bigr)^2}{R(0)-2rR(\rho) +
  r^2D(\rho)}, \ \ \rho>0\,.
$$

\begin{corollary} \label{c:most.likely}
Assume that $R$ is monotone with $R(t)\to 0$ as $t\to\infty$, and
$0<r<1$. Let 
$$
\rho_r^\ast = {\rm   argmax}_{\rho\geq 0} H_\rho(r) \,.
$$
Then $\rho_r^\ast$ is the most likely radius of the sphere to have a
hole corresponding to a factor $r$ in the center. 
\end{corollary} 
\begin{proof}
Since 
$$
\lim_{\rho\to 0} H_\rho(r)  = \lim_{\rho\to \infty} H_\rho(r) = 0\,,
$$
it follows that $\rho_r^\ast\in (0,\infty)$. Write 
$$
\rho_r= \inf\bigl\{ \rho>0:\, R(\rho)\leq  rD(\rho)\bigr\}\,.
$$
Since $0<r<1$, it follows that $\delta_r\in (0,\infty]$. Observe that
for $0<\rho<\rho_r^\ast$, by the monotonicity of $D$ and
\eqref{e:H.D}, 
\begin{equation} \label{e:two.rho}
D(\rho)> D(\rho_r^\ast) \geq H_{\rho_r^\ast}(r) \geq H_\rho(r)\,.
\end{equation}
This implies that $\rho_r^\ast\leq \rho_r$. Indeed, if this were not
the case, there would be $0<\rho<\rho_r^\ast$, for which
$R(\rho)=rD(\rho)$, and this, together with \eqref{e:H.D}, 
 would imply that $D(\rho)=H_\rho(r)$,
contradicting \eqref{e:two.rho}. 

By Theorem \ref{t:isotropic.centers} we conclude that
$W_{\rho_r^\ast}(r)=H_{\rho_r^\ast}(r)$, so it remains to prove that
$W_\rho(r)\leq H_{\rho_r^\ast}(r)$ for all $\rho\not= \rho_r^\ast$. 

However, if  $0<\rho\leq \rho_r$ , then 
$$
W_\rho(r) = H_\rho(r)\leq H_{\rho_r^\ast}(r)
$$
by the definition of $\rho_r^\ast$. On the other hand, if $\rho >
\rho_r$, then by the monotonicity of $D$,
$$
W_\rho(r) \leq D(\rho) \leq D(\rho_r) = H_{\rho_r}(r) \leq H_{\rho_r^\ast}(r)\,,
$$
and so the proof is complete. 
\end{proof}

For the covariance function $R(t)=e^{-t^2}$ the two plots of
Fig. \ref{fig:DH}  show the plot of the functions $D$ and $H(1/2)$, as
well as the optimal radius $\rho_r^\ast$ as a function of $r$. 

\begin{figure}[!ht] 
\includegraphics[width=12cm,height=5cm]{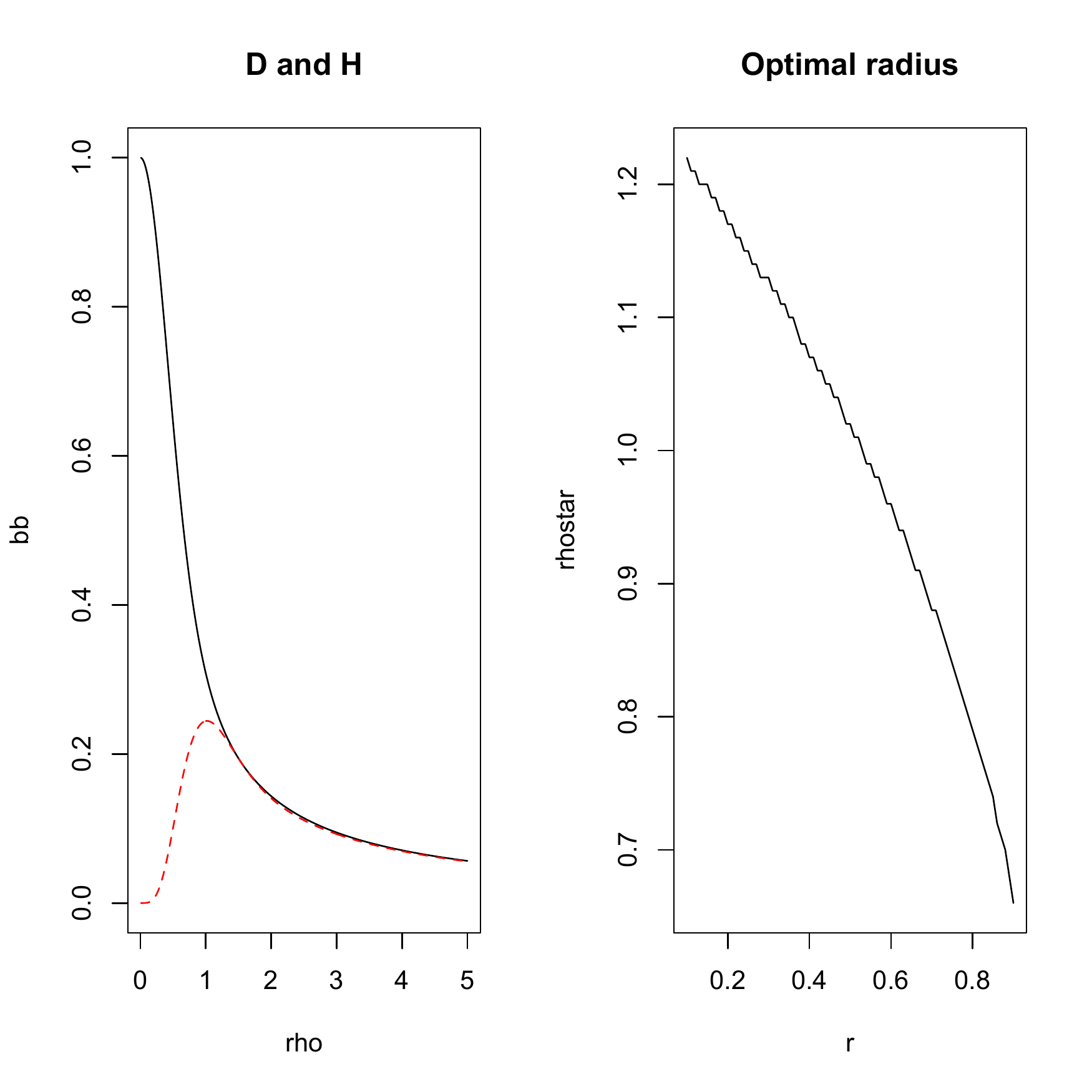}
\caption{
The functions $D(\rho)$ (solid line) and $H_\rho(r)$
  (dashed line) for $r=1/2$ (left plot) and the optimal radius
  $\rho_r^\ast$ (right plot), both for $R(t)=e^{-t^2}$}
\label{fig:DH}
\end{figure}

For the same covariance function $R(t)=e^{-t^2}$ and $r=1/2$ the plots 
of Fig. \ref{fig:shapes}  show the limiting shapes of the random field
described in Theorem \ref{t:shape.comp}. The left plot corresponds to
the sphere of radius $\rho=1$ (falling in the second case of the theorem), while
the right plot correspond to
the sphere of radius $\rho=2$ (falling in the first case of the
theorem). Note that, by the isometry of the random field, the limiting
shape is rotationally invariant. The plots, therefore, present a
section of the limiting shape along the half-axis $t_1\geq 0, \,
t_2=0$. For ease of comparison, the horizontal axis has been labeled in the
units of $t_1/\rho$, i.e. relative to the radius of the sphere.

\begin{figure}[!ht] 
\includegraphics[width=12cm,height=5cm]{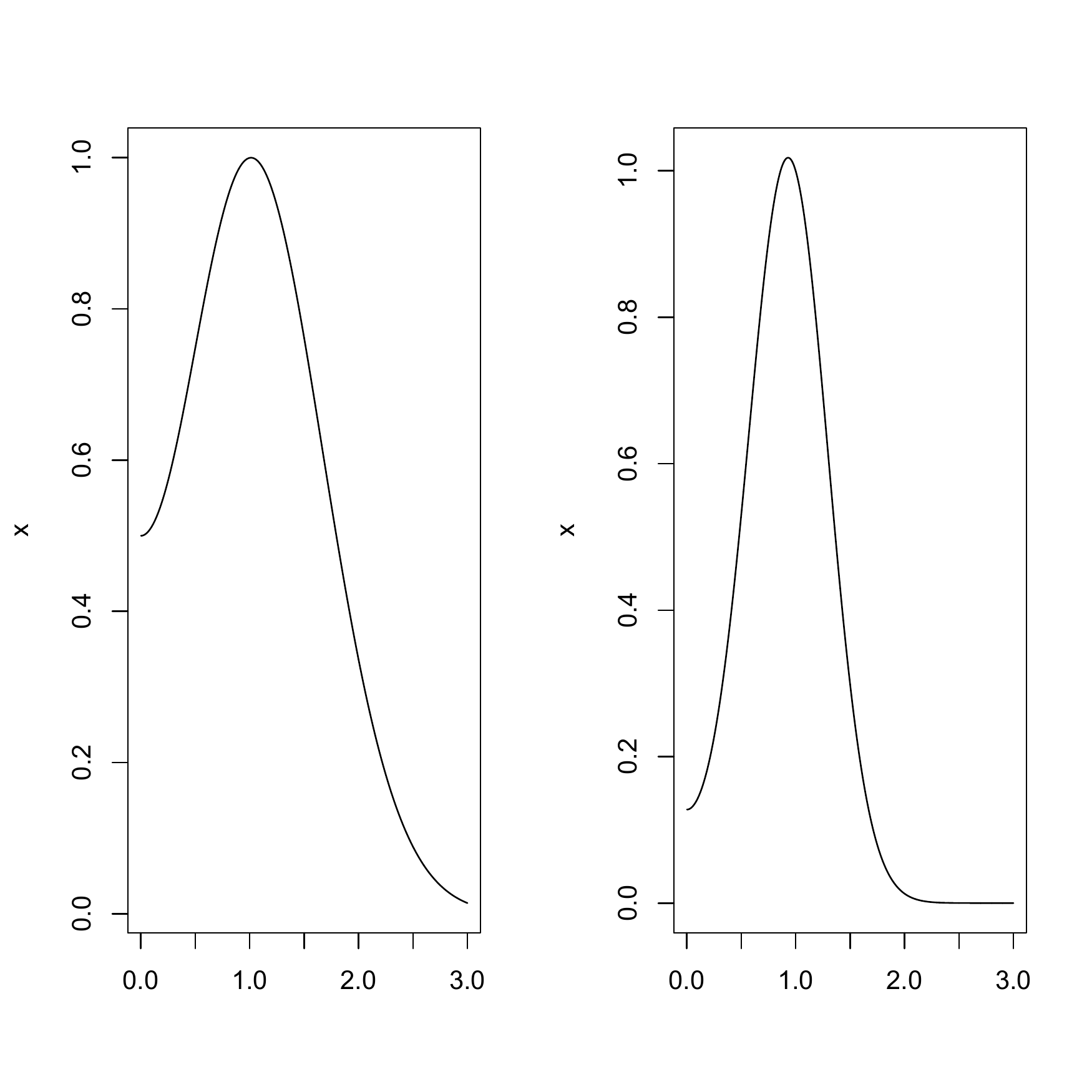}
\caption{
The  limiting shapes for $\rho=1$  (left plot) and   $\rho=2$  (right
plot), both for $r=1/2$ and $R(t)=e^{-t^2}$} 
\label{fig:shapes}
\end{figure}

We finish this section by considering the probability $\Psi_{\text
  sp}(u; r)$ in \eqref{e:spheres}. In this case, by \eqref{e:connect.R} and isotropy, 
\begin{equation} \label{e:connect.any}
D_{\cC}(r) = \min_{0\leq b\leq 1}\min_{0\leq \rho\leq D} M_\rho(r;b)\,,
\end{equation}
where $D$ is as in \eqref{e:diam}, and $M_\rho(r;b)= D_{K_1,K_2}(r) $ in
\eqref{e:fixed.pair} with $K_1$ 
being the sphere of radius $\rho$ centered at the origin, and
$K_2=\{ b\be_1\}$. Here $\be_1$ is the $d$-dimensional vector
$(1,0,\ldots, 0)$. It turns out that in many circumstances the
asymptotic behavior of the probabilities $\Psi_{\text sp; c}(u; r)$ and $\Psi_{\text
  sp}(u; r)$ is the same, at least on the logarithmic case, and so our 
analysis of the former probability applies to the latter
probability as well. 

The following result demonstrates one case when the two probabilities
are asymptotically equivalent. Assume for notational simplicity that
$R(0)=1$, and use the notation $S_1$ in place of $S_1(\bnull)$. 
 For $\rho\geq 0$, $0\leq b\leq 1$ and
$\mu\in M_1^+( S_1)$, let 
\begin{equation} \label{e:V}
V(\rho,b; \mu) = 
\end{equation} 
$$
\frac{\int_{S_1}\int_{S_1} R\bigl(\rho\| \bt_1-\bt_2\|\bigr)\, 
\mu(d\bt_1)\mu(d\bt_2) - \left( \int_{S_1} R(\rho\|\bt
  -b\be_1\|)\, \mu(d\bt)\right)^2}{ 1 -2r \int_{S_1} R(\rho\|\bt
  -b\be_1\|)\, \mu(d\bt) + r^2  \int_{S_1}\int_{S_1}
 R\bigl(\rho\| \bt_1-\bt_2\|\bigr)\,  \mu(d\bt_1)\mu(d\bt_2) }.
$$

\begin{theorem} \label{t:equiv.prob}
Let 
$$
V_*(\rho,b) = \min_{\mu\in M_1^+\bigl( S_1\bigr)} V(\rho,b;
\mu)
$$
subject to 
\begin{equation} \label{e:cond.isotr}
 \int_{S_1} R(\rho\|\bt -b\be_1\|)\, \mu(d\bt)
\geq r  \int_{S_1}\int_{S_1}
 R\bigl(\rho\| \bt_1-\bt_2\|\bigr)\,  \mu(d\bt_1)\mu(d\bt_2) \,.
\end{equation} 
If, for every $0\leq \rho\leq D$ such that $R(\rho)\geq rD(\rho)$, 
the function $V_*(\rho,b), \, 0\leq
b\leq 1$ achieves its maximum at $b=0$, then 
\begin{equation} \label{e:anywhere.exact}
\lim_{u\to\infty} \frac{1}{u^2} \log \Psi_{\text sp}(u; r) =
-\frac12 \min_{0\leq \rho\leq D} \bigl( W_\rho(r)\bigr)^{-1}\,,
\end{equation}
where $W_\rho(r)$ is defined by \eqref{e:M.r}. 
\end{theorem}
\begin{proof}
It follows from \eqref{e:connect.any}, \eqref{e:connect.center} and
Theorem \ref{t:dual.repr} that we only need to check that
$M_\rho(r)=\inf_{0\leq b\leq 1} M_\rho(r;b)$ for all $0\leq \rho\leq
D$. Notice that by \eqref{e:problem2}, \eqref{e:condition.b} and
isotropy, 
$$
M_\rho(r;b) = \left( \min\bigl( D(\rho),
  V_*(\rho,b)\bigr)\right)^{-1}\,, 
$$
where $D(\rho)$ is given in \eqref{e:def.D}. Further,
$M_\rho(r)=M_\rho(r;0)$. If $R(\rho)<rD(\rho)$,
then $V_*(\rho,0) =\infty$, so there nothing to check. If, on the
other hand, $R(\rho)\geq rD(\rho)$, then  $V_*(\rho,\cdot)$
achieves its maximum at the origin, so the claim of the theorem
follows. 
\end{proof}

The condition 
\begin{equation} \label{e:min.v}
V_*(\rho,0) =\max_{0\leq b\leq 1} V_*(\rho,b)\,,
\end{equation} 
for $0\leq \rho\leq D$ such that $R(\rho)\geq rD(\rho)$, 
deserves a discussion. We claim that this condition is implied by the
following, simpler, condition. 
\begin{equation} \label{e:min.int}
\min_{0\leq b\leq 1} \int_{S_1} R(\rho\|\bt -b\be_1\|)\, \mu_h(d\bt)
= \int_{S_1} R(\rho\|\bt \|)\, \mu_h(d\bt) = R(\rho)\,, 
\end{equation}
where $\mu_h$ is the rotation invariant probability measure on $S_1$. 

To see this let 
$R(\rho)\geq rD(\rho)$. It follows by \eqref{e:min.int} that the
constraint \eqref{e:condition.b} is satisfied for the measure $\mu_h$
and the vector $b\be_1$ for any $0\leq b\leq 1$. Therefore,
$$
V_*(\rho,b) \leq V(\rho,b; \mu_h)
= G\left( \int_{S_1} R(\rho\|\bt -b\be_1\|)\, \mu_h(d\bt)\right)\,,
$$
where 
$$
G(x) = \frac{D(\rho)-x^2}{1-2rx+r^2D(\rho)},\, \ R(\rho)\leq x\leq 1\,.
$$
Notice that
$$
G^\prime(x) =\frac{-2(x-rD(\rho))(1-rx)}{\bigl( 1-2rx+r^2D(\rho)\bigr)^2}
\leq 0\,,
$$
so that the function $G$ achieves its maximum at $x=R(\rho)$. We
conclude that 
$$
V_*(\rho,b) \leq V(\rho,0; \mu_h) = V_*(\rho,0)\,,
$$
so that \eqref{e:min.v} holds. 

Numerical experiments indicate that the condition \eqref{e:min.int}
tends to hold for values of the radius $\rho$ exceeding a certain positive
threshold. For instance, in  dimension $d=2$ for both $R(t) =
e^{-t^2}$ and $R(t) = e^{-|t|}$, this threshold is around
$\rho=1.18$. 

However, it is clear that condition \eqref{e:min.int} is not
necessary for condition \eqref{e:min.v}. In fact, for 
condition \eqref{e:min.v} to be satisfied one only needs 
a measure $\mu\in M_1^+\bigl( S_1\bigr)$ satisfying
\eqref{e:cond.isotr} such that 
\begin{equation} \label{e:good.v}
V(\rho,b; \mu) \leq V(\rho,0; \mu_h) \,,
\end{equation}
and what condition \eqref{e:min.int} guarantees is that this
measure can be taken to be the rotationally invariant measure on
$S_1$. If \eqref{e:min.int} fails, then there is no guarantee that the 
rotationally invariant measure will play the required role. 

At least in the case when the covariance function $R$ is monotone,
one can consider a measure $\mu$ that puts a point mass at the point
on the sphere closest to the point $b\be_1$. We have considered
measures $\mu\in M_1^+\bigl( S_1\bigr)$ of the form
\begin{equation} \label{e:mixture.mu}
\mu = w\delta_{{\rm sign} (b)\be_1}+ (1-w)\mu_h
\end{equation}
for some $0\leq w\leq 1$, where $\delta_a$ is, as usual, the Dirac
point mass at a point $a$. With this choice, the function $V$ in
\eqref{e:V} becomes the ratio of two quadratic functions of $w$, and
one can choose the value of $w$ that minimizes the expression, because
\eqref{e:good.v} requires us to search for as small $V$ as possible. 

In our numerical experiments we have followed an even simpler
procedure and chosen the value of $w$ that minimizes the quadratic
polynomial in the numerator of \eqref{e:V}. For the cases of $R(t) =
e^{-t^2}$ and $R(t) = e^{-|t|}$ the resulting measure $\mu$ in
\eqref{e:mixture.mu} satisfied, for all $0\leq \rho\leq D$ such that
$R(\rho)\geq rD(\rho)$, both \eqref{e:cond.isotr}  and
\eqref{e:good.v}. Therefore, in all of these cases the conclusion
\eqref{e:anywhere.exact} of Theorem \ref{t:equiv.prob} holds.

\bigskip

{\bf Acknowledgement} We are indebted to Jim Renegar of Cornell
University for useful discussions of the duality gap in convex
optimization and for drawing our attention to the paper
\cite{anderson:1983}.


\end{document}